\documentclass[11pt]{article}

\usepackage{amsmath,amssymb,amsthm,bbm}
\usepackage[utf8]{inputenc}
\usepackage{graphicx,color}
\usepackage{hyperref}
\newtheorem{thm}{Theorem}
\newtheorem{lem}{Lemma}

\newtheorem{prop}{Proposition}
\newtheorem{rem}{Remark}

\renewcommand{\thefootnote}{\dag}

\newcommand{\mb}{\mathversion{bold}}
\newcommand{\mn}{\mathversion{normal}}

\newcommand{\eps}{\varepsilon}
\renewcommand{\le}{\ell_\epsilon}
\newcommand{\lea}{\ell_{\epsilon,a}}
\newcommand{\re}{\rho_\epsilon}

\newcommand{\R}{\mathbb{R}}
\newcommand{\C}{\mathbb{C}}
\newcommand{\N}{\mathbb{N}}
\newcommand{\Z}{\mathbb{Z}}
\renewcommand{\H}{\mathbb{H}}
\newcommand{\E}{{\mathcal E}_{\varepsilon}}
\renewcommand{\P}{\mathcal{P}}
\newcommand{\Ew}{{\mathcal E}_\eps^{\rm w}}

\newcommand{\logeps}{|\!\log\eps|}
\newcommand{\GP}{{\rm $(\text{GP})_\eps^{c}$ }}
\newcommand{\gp}{(\text{GP})_\eps^{c}}
\newcommand{\LF}{{\rm (LF)}}
\newcommand{\lf}{(\text{LF})}

\def\rest{\hskip 1pt{\hbox to 10.8pt{\hfill
\vrule height 7pt width 0.4pt depth 0pt\hbox{\vrule height 0.4pt
width 7.6pt depth 0pt}\hfill}}}
\def\evalu{\hskip 1pt{\hbox to 2pt{\hfill \vrule height -6pt width 0.4pt depth
0pt}}}
\def\barint{\mathop{\vrule width 6pt height 3 pt depth -2.5pt \kern -8.8pt
\intop}}

\title{Leapfrogging vortex rings for the  three dimensional Gross-Pitaevskii
equation}
\author{{\sc Robert  L. Jerrard}  \& {\sc Didier Smets} }
\date{}
\begin{document}

\maketitle
\begin{abstract}
Leapfrogging motion of vortex rings sharing the same axis of symmetry was first
predicted by Helmholtz in his famous work on the Euler equation for incompressible fluids.
Its justification in that framework remains an open question to date. In this 
paper, we rigorously derive the 
corresponding leapfrogging motion for the axially symmetric 
three-dimensional Gross-Pitaevskii equation.   
\end{abstract}
\section{Introduction}                              %

The goal of this paper is to describe a class of cylindrically symmetric 
solutions to the three-dimensional 
Gross-Pitaevskii equation
\begin{equation*}\label{eq:GP}
i\partial_t u - \Delta u = \frac{1}{\eps^2}u(1-|u|^2)
\end{equation*}
for a complex-valued function $u:\:\R^3\times \R \to \C$. In the regime 
which we shall describe, it turns out that the Gross-Pitaevskii 
equation bears some 
resemblance with the Euler equation for 
flows 
of 
incompressible fluids
$$
\left\{ 
\begin{array}{l}
\partial_t v + (v\cdot \nabla) v = -\nabla p\\
{\rm div}\, v = 0,
\end{array}
\right.
$$
where $v:\:\R^3\times \R \to \R^3$ is the velocity field and $p:\: 
\R^3\times \R \to \R$ is the pressure field.  In this analogy,  
the role of the velocity $v$ is played by the current\footnote{For $y\in \C$ 
and $z=(z_1,\cdots,z_k) \in \C^k$ we write $(y,z):=\big({\rm 
Re}(y\bar z_1),\cdots,{\rm Re}(y\bar z_k)\big) \in \R^k$ and $y\times z := 
(iy,z)$.}
$$
j(u) := u \times \nabla u = (iu,\nabla u) = {\rm Re} (u \nabla \bar{u} )
$$
and the vorticity field $\omega:={\rm curl}\, v$ therefore corresponds, up to a 
factor of two, to the Jacobian 
$$
J(u) := \frac{1}{2}{\rm curl}\, j(u) = \Big( \partial_2 u \times \partial_3 u, 
\partial_3 u \times \partial_2 u, \partial_1 u \times \partial_2u\Big). 
$$

In his celebrated work \cite{Hel1,Hel2} on the Euler equation, Helmholtz 
considered
with great attention the situation where the vorticity field $\omega$ is
concentrated in a ``circular vortex-filament of very small section'', a 
thin vortex ring.  
A central question in Helmholtz's work, as far as dynamics is concerned, is 
related
to the possible forms of stability of the family of such vortex 
rings, allowing a change in time of cross-section, radius, position  
or even possibly of inner profile, and a description of these evolutions.  
When only one vortex-filament is present, Helmholtz's conclusions are :  
\begin{quote}\it{
 Hence in a circular vortex-filament of very small section in an indefinitely
extended fluid, the center of gravity of the section has, from the commencement, 
an
approximately constant and very great velocity parallel to the axis of the
vortex-ring, and this is directed towards the side to which the fluid flows 
through
the ring.}
\end{quote}
Instead, when two vortex-filaments interact,  Helmholtz predicts the
following :
\begin{quote}{\it
We can now see generally how two ring-formed vortex-filaments
having the same axis would mutually affect each other, since each, in addition 
to its
proper motion, has that of its elements of fluid as produced by the other. If 
they
have the same direction of rotation, they travel in the same direction; the 
foremost
widens and travels more slowly, the pursuer shrinks and travels faster till 
finally,
if their velocities are not too different, it overtakes the first and penetrates 
it.
Then the same game goes on in the opposite order, so that the rings pass through 
each
other alternately.}
\end{quote}

The motion described by Helmholtz, and illustrated in Figure \ref{fig:one} 
 below, 
is often termed {\it leapfrogging} in the fluid mechanics community. Even 
though it has been widely studied since Helmholtz, as far as we know it has 
not been mathematically justified in the context of the Euler equation, even in 
the axi-symmetric case without swirl\footnote{We refer to \cite{CaMa,MaNe} for 
some attempts in that direction, and an account of the difficulties.}.  As a 
matter of fact, the interaction 
leading to the leapfrogging motion is somehow borderline in strength compared 
to the stability of isolated vortex rings.

\begin{figure}[ht!]\label{fig:one}
\centering
\includegraphics[width=0.95\textwidth]{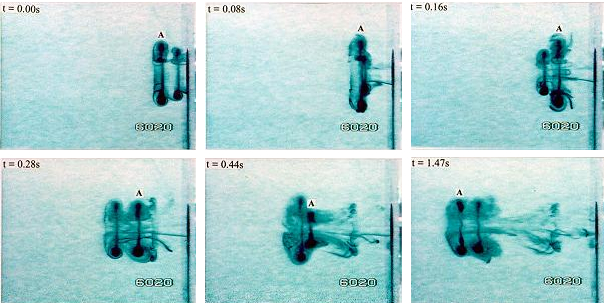}
\caption{\copyright T.T. Lim, Phys. of Fluids, Vol. 9}
\label{overflow}
\end{figure}

Our main results in this paper, Theorem \ref{thm:main_asympt} and 
\ref{thm:main} below, provide a 
mathematical justification to the leapfrogging motion of two or more vortex 
rings in
the context of the axi-symmetric three-dimensional Gross-Pitaevskii equation. 

\medskip
\subsection{Reference vortex rings}

A well-known particularity of the Gross-Pitaevskii equation is that vortex ring  
intensities are necessarily quantized. For stability reasons, we only consider 
simply quantized rings. 

\medskip

Let $\mathcal C$ be 
a smooth oriented closed curve in $\R^3$ and let $\vec{\mathcal J}$ be the 
vector
distribution corresponding to $2\pi$ times the circulation along $\mathcal C$, 
namely 
$$
\langle \vec{\mathcal{J}},\vec X\rangle = 2\pi \int_{\mathcal C} \vec X \cdot 
\vec
\tau\qquad \forall \vec X \in \mathcal{D}(\R^3,\R^3),
$$
where $\vec \tau$ is the tangent vector to $\mathcal C.$ 
To the ``current density'' $\vec{\mathcal J}$ is associated the ``induction'' 
$\vec
B$, which satisfies the equations
$$
{\rm div }(\vec B)=0,\qquad {\rm curl} (\vec B) = \vec{\mathcal J} \qquad 
\text{in }
\R^3,
$$
and is obtained from $\vec {\mathcal J}$ by the Biot-Savart law. To $\vec B$ is 
then associated a vector potential
$\vec A$, which satisfies 
$$
{\rm div }(\vec A)=0,\qquad {\rm curl }(\vec A) = \vec B \qquad \text{in } \R^3,
$$
so that 
$$
-\Delta \vec A = {\rm curl}\, {\rm curl} (\vec A) = \vec{\mathcal J} \qquad 
\text{in
} \R^3.
$$
Since we only consider axi-symmetric configurations in this paper, we let 
 $\H$ to be the half-space $\{(r,z)\ | \ r>0 , z \in \R\}$ and we denote by  
$r(\cdot)$ and $z(\cdot)$ the coordinate 
functions in $\H.$
For $a\in \H$, let $\mathcal C_a$ be the circle of
radius $r(a)$ parallel to the 
$xy$-plane in $\R^3$, centered at the point $(0,0,z(a))$, and oriented so that 
its binormal vector points towards the positive $z$-axis. By cylindrical 
symmetry, we may write the corresponding vector potential as 
$$
\vec A_a \equiv A_a(r,z) \vec{e}_\theta. 
$$
The expression of the vector Laplacian in cylindrical coordinates yields the
equation for the scalar function $A_a$ :
$$
\left\{
\begin{array}{ll}
\displaystyle -\left(\partial^2_r +\frac{1}{r}\partial_r - \frac{1}{r^2} +
\partial^2_z\right) A_a =  2\pi \delta_{a} &\qquad \text{in } \H\\
A_a   = 0 & \qquad \text{on } \partial\H,
\end{array}
\right.
$$  
or equivalently
$$
\left\{
\begin{array}{ll}
\displaystyle -{\rm div} \left(\frac{1}{r} \nabla \left( r A_a\right)
\right) =  2\pi \delta_{a} &\qquad \text{in } \H\\
A_a  = 0 & \qquad \text{on } \partial\H,
\end{array}
\right.
$$  
which can be integrated explicitly in terms of complete elliptic 
integrals\footnote{The integration is actually simpler in the 
original cartesian
coordinates. A classical reference is the book of Jackson \cite{Jackson}, 
an extended analysis can be found in the 1893 paper of Dyson \cite{Dyson}. 
See Appendix A for some details.}. 

Up to a constant phase factor, there exists a unique unimodular map $u^*_a
 \in \mathcal{C}^\infty(\H \setminus \{a\},S^1)\cap W^{1,1}_{\rm
loc}(\H,S^1)$ such that
$$
r (iu^*_a,\nabla u^*_a) = rj(u^*_a) = -\nabla^\perp (rA_a).
$$   
In the sense of distributions in $\H$, we have
$$
\left\{
\begin{array}{ll}   
\displaystyle {\rm div}(rj(u^*_a)) & = 0\\
\displaystyle {\rm curl}(j(u^*_a)) & = 2\pi \delta_{a},
\end{array}
\right.
$$
and the function $u^*_a$ corresponds therefore to a {\it singular} vortex ring. 
In order to describe a {\it reference} vortex ring for the 
Gross-Pitaevskii 
equation, we shall make the notion of core more precise. In $\R^2$, the 
Gross-Pitaevskii equation possesses a distinguished stationary solution called 
vortex : in polar coordinates, it has the special form 
$$
u_{\eps}(r,\theta) = f_{\eps} 
(r)\exp(i\theta)
$$    
where the profile $f_{\eps} : \R^+ \to [0,1]$ satisfies $f_{\eps}(0)=0$,  
$f_{\eps}(+\infty)=1,$ and
$$
\partial_{rr}f_{\eps}  + \frac{1}{r}\partial_r f_{\eps} - 
\frac{1}{r^2}f_{ \eps} + 
\frac{1}{\eps^2} f_{\eps} (1-f_{\eps}^2) = 0.
$$
Notice that $\eps$ has the dimension of a length, and since by scaling 
$f_{\eps} (r)=f_{1}(\tfrac{r}{\eps})$ it is the 
characteristic length of the core. 

The reference vortex ring
associated to the point $a \in \H$ is defined to be 
\begin{equation*}
u_{\eps,a}^*(r,z) = f_\eps\big(\|(r,z)-a\|\big) u^*_a(r,z).
\end{equation*}
More generally, when $a = \{a_1,\cdots,a_n\}$ is a family of $n$ distinct 
points in $\H$, we  set 
$$
u_{a}^*(r,z) := \prod_{k=1}^n u^*_{a_k}(r,z),\quad\text{and}\quad  
u_{\eps,a}^*(r,z) := \prod_{k=1}^n u^*_{\eps,a_k}(r,z),
$$
where the products are meant in $\C.$ The field 
$u_{\eps,a}^*$ hence corresponds to a collection of $n$ 
reference vortex rings (sharing the same axis and oriented in the same 
direction), and is the typical kind of object which we shall study the evolution 
of. It can be shown that 
\begin{equation*}\label{eq:concjac*}
 \Big\|J u_{\eps,a}^* - \pi \sum_{i=1}^n 
\delta_{a_{i}}\Big\|_{\dot W^{-1,1}(\H)} = \ 
O(\eps)\quad\text{as } \eps \to 0,
\end{equation*}
where here and in the sequel, for 
a complex function $u$ on $\H$ we denote by $Ju$ its jacobian function 
$Ju= \partial_ru \times \partial_z u.$

\subsection{The system of leapfrogging}

Being an exact collection of (or even a single) reference vortex 
rings is not a 
property which is preserved by the flow of the Gross-Pitaevskii 
equation\footnote{Exact traveling waves having the form of vortex rings have 
been constructed in \cite{BOS-VR}, these are very similar in shape but not 
exactly equal to reference vortex rings.}. To carry out our analysis, 
 we rely mainly on the energy density and the current density.  
For cylindrically symmetric solutions $u\equiv u(r,z,t)$,  
the Gross-Pitaevskii equation writes  
$$
\left\{
\begin{array}{lll}
\displaystyle ir\partial_t u - {\rm div}(r\nabla u) =
\frac{1}{\eps^2}ru(1-|u|^2) &\text{in } &\H \times \R ,\\
\displaystyle \partial_r u = 0 & \text{on } 
&\partial\H\times\R.
\end{array}
\right.\leqno{\gp}
$$
Equation \GP is an hamiltonian flow for the (weighted) Ginzburg-Landau energy
$$
\Ew(u) := \int_{\H} \left( \frac{|\nabla u|^2}{2} +
\frac{(1-|u|^2)}{4\eps^2}\right) r\, dr dz,
$$
and the Cauchy problem is known to be well-posed for initial data with finite 
energy. Classical computations leads to the estimate :

\begin{lem}\label{lem:energy}
It holds
$$
 \Ew\big(u^*_{\eps,a}\big) = \sum_{i=1}^n r(a_i) \Big[ \pi 
\log\big(\tfrac{r(a_i)}{\eps}\big) + \gamma + \pi\big(3\log(2)-2\big) 
+ \pi\sum_{j\neq i}A_{a_j}(a_i) +   
O\big((\tfrac{\eps}{\rho_a})^{\tfrac23}\log^2(\tfrac{\eps}{\rho_a}) 
\big)\Big], 
$$
where 
\begin{equation}\label{eq:defrhoa}
 \rho_a :=\frac14\min\Big( \min_{i\neq j}|a_i-a_j|, \min_i 
r(a_i)\Big).
\end{equation}
\end{lem}
\noindent
In Lemma \ref{lem:energy}, the constant $\gamma$ is defined by (see 
\cite{BBH})
$$
\gamma := \liminf_{\eps \to 0} \Big[ \E(v_\eps,B_1) - \pi \logeps 
\Big] \text{ with }  v_\eps \in H^1(B_1,\C) \text{ and } v_\eps(z)=z \text{ on 
} \partial B_1,  
$$
where $B_1$ is the unit disk in $\R^2$ and where for an open subset $\Omega 
\subset \R^2$ and $u \in H^1_{\rm 
loc}(\Omega,\C)$ we denote the unweighted two-dimensional Ginzburg-Landau 
energy of $u$ in $\Omega$ by
$$
\E(u,\Omega) =\int_\Omega e_\eps(u) d\mathcal{L}^2 := \int_{\Omega} \left( 
\frac{|\nabla u|^2}{2} +
\frac{(1-|u|^2)}{4\eps^2}\right) d\mathcal{L}^2.
$$

In light of Lemma \ref{lem:energy}, we define the quantity
$$
H_\eps(a_1,\cdots,a_n) := \sum_{i=1}^n r(a_i) \Big[ \pi 
\log\big(\tfrac{r(a_i)}{\eps}\big) + \gamma + \pi\big(3\log(2)-2\big) 
+ \pi\sum_{j\neq i}A_{a_j}(a_i)\Big],
$$
and we consider the associated hamiltonian system 
$$
\hspace{1.5cm} \dot{a}_i(s) = \ \frac{1}{\pi\logeps}
\mathbb{J}\nabla_{a_i}H_\eps\big(a_1(s),\cdots,a_n(s)\big),  \hspace{2.5cm} 
{i=1,\cdots,n,} \leqno{\lf_\eps}
$$
where, with a slight abuse of notation, 
$$
\mathbb J := \begin{pmatrix}0 & -\frac{1}{r(a_i)} \\ 
\frac{1}{r(a_i)} & 0\end{pmatrix}.
$$
In addition to the hamiltonian $H_\eps,$ the  
system $\lf_\eps$ also conserves the momentum
$$
P(a_1,\cdots,a_n) := \pi \sum_{k=1}^n r^2(a_k),
$$
which may be interpreted as the total area of the disks determined by the vortex
rings. As a matter of fact, note also that
$$
\P(u_{\eps,a}^*) := \int_\H Ju_{\eps,a}^*\, 
r\, 
drdz = \pi \sum_{k=1}^n r^2(a_k) + o(1),
$$
as $\eps \to 0,$ and that, at least formally, the momentum  $\P$ is a 
conserved quantity for \GP.

When $n=2$, the system $\lf_\eps$ may be analyzed in great details. Since $P$ 
is conserved and since $H_\eps$ is invariant by a joint translation of both  
rings in the $z$ direction, it is classical to introduce the variables 
$(\eta,\xi)$ 
by
$$
\left\{
\begin{array}{l}
r^2(a_1)=\frac{P}{2} - \eta\\
r^2(a_2)=\frac{P}{2} + \eta
\end{array}
\right.,
\qquad
\xi=z(a_1)-z(a_2),
$$
and to draw the level curves of the function $H_\eps$ in those two 
real variables, the momentum $P$ being considered as a parameter. 

The next 
figure illustrates the global behavior of the phase portrait, with three 
distinct regions which we have called ``{\it pass through}'', ``{\it attract 
then repel}'' 
and ``{\it leapfrogging}''. 
The leapfrogging region corresponds to the central part, where all solutions 
are periodic in time; its 
interpretation was discussed earlier in this introduction. In the 
pass through region, the first vortex ring always remains the smallest, hence 
quickest, of the two vortex rings : being initially located below the second 
vortex ring on the z-axis it first catches 
up, then  passes inside the second  and finally gets away in front of 
it\footnote{A similar situation is described by Hicks \cite{Hicks} for a 
simplified vortex model introduced by Love \cite{Love} in 1894.}.  
Instead, in the attract then repel region the first vortex ring 
initially starts to catch up, but doing so its circular radius increases 
whereas the one of the second vortex ring decreases, up to a point where  both 
vortex rings have the same radius and the first still lag behind the second. 
From that point on, the first one has a larger radius than the second, and 
therefore the second increases its lead indefinitely. The behavior 
in those 
last two regions is actually very much reminiscent of two-solitons 
interactions in the Korteweg - de Vries equation, in particular the speeds at 
plus and minus 
infinity in time are equal or exchanged. Notice also that the two points at the 
common boundary of the three regions correspond, up to labeling, to the same 
situation : two vortex rings travel with the same constant speed at 
 a special mutual 
distance\renewcommand{\thefootnote}{\ddag}\footnote{We stress that 
this holds at the level of the system $\lf_\eps$, we do not know whether such 
special solutions exist 
at the level of equation $\gp$.}\renewcommand{\thefootnote}{\dag}.    
\begin{figure}[h!]\label{fig:two}
\centering
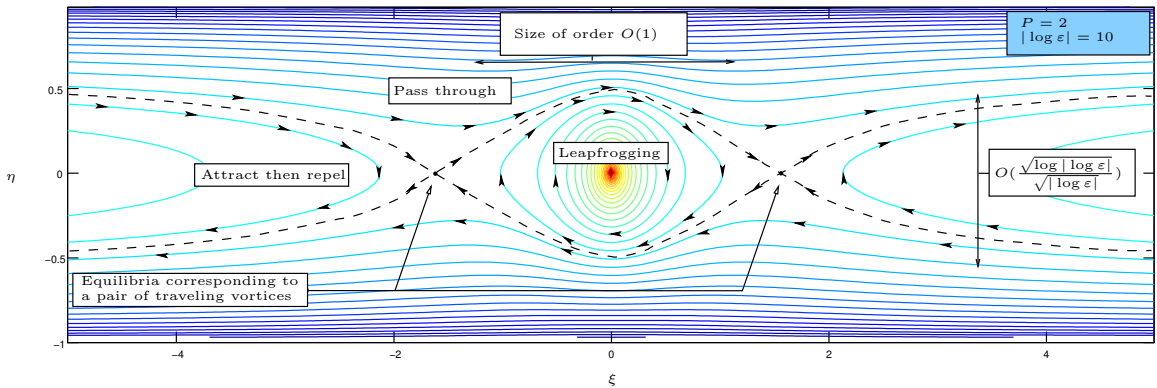
\caption{Phase portrait of the system $\lf_\eps$ for two vortex rings}
\end{figure}

The typical size of the leapfrogging region is also described in the figure. 
In particular, it shrinks and becomes more flat as $\eps$ decreases towards  
zero.

\subsection{Statement of the main results}

We present two results in this section. The first one follows rather easily  
from the second, but its statement has the advantage of being 
somewhat simpler. On the other hand, it involves a limiting procedure $\eps \to 
0,$ whereas the second one is valid for small but fixed values of $\eps.$  

\smallskip

In order to state those results, and in view of the size of the leapfrogging 
region mentioned at the end of the previous subsection, we fix some 
$(r_0,z_0)\in \H$, 
an integer $n\geq 1$, and $n$ distinct points $b_1^0,\cdots,b_n^0$ in $\R^2$. 
The initial positions of the cores of the vortex rings are then set to be 
$$
a_{i,\eps}^0 := \Big(r_0+ \frac{r(b_i^0)}{\sqrt{\logeps}} , 
z_0 + \frac{z(b_i^0)}{\sqrt{\logeps}} \Big),\qquad i=1,\cdots,n.
$$
As a matter of fact, this is the 
appropriate scaling for which {\bf relative} self-motion and  interactions between 
vortex-rings are of the same magnitude. 
In any scaling in which $a^0_{i,\eps} - (r_0,z_0) = o(1)$ as $\eps \to 0$ 
for all i, the ``leading-order'' vortex motion is
expected to be a translation with constant velocity $1/r_0$ in the vertical
direction and in the rescaled time. The above scaling is the appropriate one for which, in
the next-order correction, the difference in the self-motion speeds (due to different
values of the radii at the next order) and interaction between vortices
are of the same magnitude. In the case of two vortices, for example,
this will give rise to small-scale periodic corrections to a leading-order
translation, which is the signature of "leapfrogging".

Note that $a_{i,\eps}^0 \in \H$ provided $\eps$ is sufficiently small, which we 
assume throughout. Concerning their evolution, we consider the solution to 
the Cauchy problem 
for the system of ordinary differential equations
$$
\qquad \quad
\left\{
\begin{array}{rl}
\displaystyle \dot{b}_i(s) &= \ \sum_{j\neq i}
\frac{(b_i(s)-b_j(s))^\perp}{\|b_i(s)-b_j(s)\|^2} -
\frac{r(b_i(s))}{r_0^2}\begin{pmatrix}0\\1\end{pmatrix}\\
 \displaystyle b_i(0) &= \ b_i^0
\end{array}
\right. \qquad \qquad i=1,\cdots,n,\leqno{\lf}
$$
 and we finally set 
\begin{equation}\label{eq:ais}
 a_{i,\eps}(s) := \Big(r_0+ \frac{r(b_i(s))}{\sqrt{\logeps}} , 
z_0 + \frac{s}{r_0} + \frac{z(b_i(s))}{\sqrt{\logeps}} \Big).
\end{equation}
System {\lf} and \eqref{eq:ais} describe the main order asymptotic of 
$\lf_\eps$ in the leapfrogging region, after a proper rescaling in time. 

\medskip
We will prove

\begin{thm}\label{thm:main_asympt}  Let $(u_\eps^0)_{\eps>0}$ 
be a family of initial
data for \GP such that 
\begin{equation}\label{eq:concjac0}
 \Big\|J u_\eps^0 - \pi \sum_{i=1}^n 
\delta_{a_{i,\eps}^0}\Big\|_{\dot W^{-1,1}(\Omega)} = \ 
o\Big(\frac{1}{\logeps}\Big),
\end{equation}
as $\eps \to 0$, for any open subset $\Omega$ strongly included in 
$\H$. Assume also that
\begin{equation}\label{eq:excess0}
 \Ew(u_\eps^0) \leq 
H_\eps(a_{1,\eps}^0,\cdots,a_{n,\eps}^0) + 
o(1), \text{ as } \eps \to 0.
\end{equation}
Then, for every $s \in \R$ and every open subset $\Omega$
strongly included in $\H$ we have 
\begin{equation}\label{eq:concjact}
 \Big\|J u_\eps^s - \pi \sum_{i=1}^n 
\delta_{ a_{i,\eps}(s)}\Big\|_{\dot W^{-1,1}(\Omega)} =
o\Big(\frac{1}{\sqrt{\logeps}}\Big)
\end{equation}
where we denote by $u_\eps^s$ the solution of {\GP} with initial datum 
$u_\eps^0$ and evaluated at time $t=s/\logeps,$ and  
where the points $a_{i,\eps}(s)$ are defined in 
\eqref{eq:ais} through the solution of the system \LF.
\end{thm}

In the statement of Theorem \ref{thm:main_asympt}, the $\dot W^{-1,1}$ norm is defined by 
$$
\|\mu\|_{\dot W^{-1,1}(\Omega)} = {\rm sup}\left\{ \int \varphi \, d\mu, \ \varphi \in W^{1,\infty}_0(\Omega),\ \|\nabla \varphi\|_\infty \leq 1\right\}.  
$$

\begin{rem}\label{rem:1}
    Asymptotic formulas for the potential vectors $A_{a_i}$ (see Appendix A) lead to the 
equivalence 
\begin{equation}\label{eq:Hepseq}
H_\eps(a_{1,\eps},\cdots,a_{n,\eps}) = \Gamma_\eps(r_0,n) + 
W_{\eps,r_0}(b_1,\cdots,b_n) + o(1) \text{ as } \eps \to 0,
\end{equation}
where 
$
\Gamma_\eps(r_0,n) = nr_0( \pi |\log\eps| + \gamma + \pi n\log r_0 + \pi n 
(3\log2-2) + \pi \tfrac{n-1}{2}\log|\log\eps|)
$
and 
\begin{equation}\label{eq:Weps}
W_{\eps,r_0}(b_1,\cdots,b_n) = \pi \sum_{i=1}^n r(b_i) \sqrt{\logeps} - \pi r_0 
\sum_{i\neq j}
\log|b_i-b_j|.
\end{equation}
Also, expansion of the squares leads directly to
$$
P(a_{1,\eps},\cdots,a_{n,\eps}) = \pi n r_0^2 + 2\pi r_0\sum_{i=1}^n 
\frac{r(b_i)}{\sqrt{\logeps}} + \pi \sum_{i=1}^n \frac{r(b_i)^2}{\logeps},
$$
and therefore 
$$
\Big(H_\eps- \frac{\logeps}{2r_0}P\Big)(a_{1,\eps},\cdots,a_{n,\eps}) = 
-\frac{\pi}{2} nr_0 \logeps +\Gamma_\eps(r_0,n)+ \pi r_0 W(b_1,\cdots,b_n) + o(1),\quad\text{as } 
\eps \to 0, 
$$
where 
$$
W(b_1,\cdots,b_n) :=  -  
\sum_{i\neq j}
\log|b_i-b_j|- \frac{1}{2r_0^2} \sum_{i=1}^n r(b_i)^2.
$$
The function $W$, which does not depend upon $\eps$, is precisely the 
hamiltonian for the system \LF. A second quantity preserved by {\LF} is given
by $Q(b_1,\cdots,b_n):=\sum_{i=1}^n r(b_i).$ When $n=2$, all 
the solutions are {\LF} are periodic in time.
\end{rem}

We will now state  a quantitative version of Theorem 
\ref{thm:main_asympt} which holds for small but fixed values of $\eps,$ 
not just asymptotically as $\eps \to 0.$ 
We fix  positive constants $K_0$ and $r_0$ and we consider an arbitrary  solution 
$a_\eps(s)\equiv\{a_{i,\eps}(s)\}_{1\leq i\leq n}$ of the system $\lf_\eps$ on some 
time interval $[0,S_0]$, $S_0\geq 0,$ which we assume to satisfy  
\begin{equation}\label{eq:bonneechelle}
    \begin{array}{c}
        \displaystyle \frac{K_0^{-1} }{\sqrt{\logeps}}\leq \min_{s\in [0,S_0]}\min_{i\neq j} 
    |a_{i,\eps}(s)-a_{j,\eps}(s)| 
    \leq \max_{s\in [0,S_0]}\max_{i\neq j} |a_{i,\eps}(s)-a_{j,\eps}(s)| \leq \frac{K_0}{\sqrt{\logeps}}\\
\displaystyle \frac{r_0}{2} \leq \min_{s\in [0,S_0]}\min_{i} r(a_{i,\eps}(s))
\leq \max_{s\in [0,S_0]}\max_{i} r(a_{i,\eps}(s))\leq 2r_0.
\end{array}
\end{equation}

We define the localization scale  
\begin{equation}\label{eq:local0}
 r_a^0 := \Big\|J u_\eps^0 - \pi \sum_{i=1}^n 
 \delta_{a_{i,\eps}^0}\Big\|_{\dot W^{-1,1}(\Omega_0)},
\end{equation}
where $\Omega_0:= \{r\geq \frac{r_0}{4}\}$, and the excess energy 
\begin{equation}\label{eq:exc0}
\Sigma^0 := \left[\Ew(u_\eps^0) - 
H_{\eps}(a_{1,\eps}^0,\cdots,a_{n,\eps}^0)\right]^+
\end{equation}
at the initial time.

\begin{thm}\label{thm:main}
Let $a_\eps(s)\equiv\{a_{i,\eps}(s)\}_{1\leq i\leq n}$ be a solution of the system $\lf_\eps$
on some time interval $[0,S_0]$, $S_0\geq 0,$ which satisfies \eqref{eq:bonneechelle}. There exist positive 
numbers $\eps_0$, $\sigma_0$ and $C_0$, depending only on $r_0$, $n$,    
$K_0$ and $S_0$ with the following properties. Assume that $0 < \eps \leq 
\eps_0$ and that
\begin{equation}\label{eq:excesspetit}
r_a^0\logeps + \Sigma^0 \leq \sigma_0,
\end{equation}
then  
$$
\Big\|J u_\eps^s - \pi \sum_{i=1}^n 
\delta_{a_{i,\eps}(s)}\Big\|_{\dot W^{-1,1}(\Omega_0)} \leq C_0\left( r_a0 + \frac{\Sigma^0}{\sqrt{\logeps}}
+ \frac{C_\delta}{\logeps^{1-\delta}}\right)e^{C_0s},
$$
for every $s \in [0,S_0],$ where $\delta>0$ can be chosen arbitrarily small. 
\end{thm}

\medskip

To finish this introduction, let us mention that we have not analyzed the 
convergence of \GP towards $\lf_\eps$ in the ``pass through'' and ``attract 
then repel'' regions. It is conceivable, yet probably difficult, to 
obtain closeness estimates valid for all times in those cases, reminiscent of
what is sometimes called orbital stability of multi-solitons, e.g. in the 
Korteweg~-~de~Vries equation \cite{MaMeTs} or the 1D Gross-Pitaevskii 
equation \cite{BeGrSm}. One would have to deal with algebraic rather than 
exponential interaction estimates.

Also, having in mind the initial question related to the Euler 
equation, let us mention that one crucial advantage in the analysis of the 
Gross-Pitaevskii equation is that it has an inherent core localization scale 
$\eps.$ On the other hand, Euler velocity fields are divergence free, whereas
Gross-Pitaevskii ones only have small divergence when averaged in time.  
Analysis of leapfrogging for the Euler equation would therefore probably require 
a different strategy. 

\bigskip
\noindent{\bf Acknowledgements.} The research of RLJ was partially supported by the National Science and Engineering Council of Canada under operating grant 261955. The research of DS was
partially supported by the Agence Nationale de la Recherche through the project ANR-14-CE25-0009-01.

\section{Strategy for the proofs}\label{sect:strategy}

The overall strategy follows many of the lines which we adopted in our 
prior work \cite{JeSm3} on the inhomogeneous Gross-Pitaevskii equation\footnote{Another
work on the 2D inhomogeneous GP equation is a recent preprint of Kurzke et al \cite{KuMaSp}, which studies a situation where the inhomogeneity and its derivatives are of order $\logeps^{-1}.$ This is critical in the sense that interaction of vortices with the background potential and with each other are of the same order of magnitude. In the present work, by contrast, critical coupling occurs in hard-to-resolve corrections to the leading-order dynamics.} The 
effort is actually focused on Theorem \ref{thm:main} first, Theorem 
\ref{thm:main_asympt} can be deduced from it rather directly. The essential new ingredients
with respect to \cite{JeSm3} are refined approximation estimates (mainly Proposition \ref{prop:stronglocal})
and the key observation in Proposition \ref{cor:apprixgood}.

\subsection{Localisation, excess energy and approximation by a reference field}\label{subsect:local}

In this section we present arguments which are not directly related to the time evolution but
only to some assumptions on the energy density and on the Jacobian of a function $u$.  
In rough terms, we assume that $u$ is known a priori to satisfy
some localisation estimates and some energy upper bounds, and we will show, by combining them
together, that under a certain approximation threshold this can be improved by a large
amount, without any further assumption. 

In order to state quantitative results,  
we assume here that $\{a_i\}_{1\leq i \leq n}$ is
a collection of points in $\H$ such that 
\begin{equation}\label{eq:H1}
    \begin{array}{c}
        \displaystyle \frac{8}{\logeps} \leq \min_{i\neq j} |a_{i}-a_{j}| 
\leq \max_{i\neq j} |a_{i}-a_{j}| \leq K_1\\
\displaystyle \frac{r_0}{2} \leq r(a_{i})
\leq \max_{i} r(a_{i})\leq 2r_0.
\end{array}\tag{${H_1}$}
\end{equation}

We assume next that  $u \in H^1_{\rm 
loc}(\H,\C)$ is such that its Jacobian $Ju$ satisfies the rough 
localisation estimate
\begin{equation}\label{eq:defra}
r_a := \|Ju-\pi \sum_{i=1}^n \delta_{a_i}\|_{\dot 
W^{-1,1}(\Omega_0)}< \frac{\rho_a}{4},
\end{equation}
where $\rho_a$ is defined in \eqref{eq:defrhoa}.
We finally define the excess energy relative to those 
points,
\begin{equation}\label{eq:defSigma}
\Sigma_a:= \left[ \Ew(u)-H_\eps(a_1,\cdots,a_n)\right]^+.
\end{equation}
We will show that if $r_a$ and $\Sigma_a$ are not too large then actually a 
much better form of localisation holds. 


\begin{prop}\label{prop:stronglocal}
    Under the assumption $(H_1)$ and \eqref{eq:defra}, there exist constants $\eps_1,\sigma_1,C_1>0,$ 
depending only on $n$, $r_0$ and $K_1$, with the following properties. If 
$\eps \leq \eps_1$ and 
\begin{equation}\label{eq:boundinit}
    \Sigma_a^r := \Sigma_a +r_a\logeps\leq \sigma_1\logeps,
\end{equation}
then there exist $\xi_1,\cdots,\xi_n$ in $\H$ such 
that  
\begin{equation}\label{eq:dsd1}
    \|Ju-\pi\sum_{i=1}^n \delta_{\xi_i}\|_{\dot W^{-1,1}\left(\left\{r\geq C_1(\Sigma_\xi+1)/\logeps \right\}
    \right)} \leq C_1 \eps \logeps^{C_1}e^{C_1\Sigma_a^r},
\end{equation}
and
\begin{equation}\label{eq:dsd2}
\int_{\H \setminus \cup_i B(\xi_i,\eps^\frac23)} \!\!  
r \left[ e_\eps(|u|) + 
\big|\frac{j(u)}{|u|}-j(u_\xi^*)\big|^2\right]  \leq 
C_1\big(\Sigma_\xi + 
\eps^\frac13\logeps^{C_1}e^{C_1\Sigma_a^r }\big),
\end{equation}
where we have written
$$
\Sigma_\xi :=   \left[ \Ew(u)-H_\eps(\xi_1,\cdots,\xi_n)\right]^+.    
$$
Moreover,
\begin{equation}\label{eq:sigmaxicontrolbad}
\Sigma_\xi \leq \Sigma_a + C_1r_a\logeps + C_1 \eps \logeps^{C_1}e^{C_1\Sigma_a^r},
\end{equation}
and the values of $\eps_1$ and $\sigma_1$ are chosen sufficiently small so that
\begin{equation}\label{eq:securitybounds}
    C_1\logeps^{C_1}e^{C_1\sigma_1\logeps} \leq \eps^{-\frac16},\qquad\text{and}\qquad C_1(\Sigma_\xi+1)/\logeps \leq \frac{r_0}{4}
\end{equation}
whenever $\eps\leq \eps_1.$ \end{prop}

\begin{rem}\label{rem:crucial}
    It is tempting to simplify somewhat the statement of Proposition \ref{prop:stronglocal} by replacing
    the term $\Sigma_\xi$ in the right-hand side of \eqref{eq:dsd2} by $\Sigma_a^r$ (in view of \eqref{eq:sigmaxicontrolbad} 
    this would be correct up to a possible change of $C_1$), and hence obtain error bounds that only depend
    on the input data. Yet, it turns out that \eqref{eq:sigmaxicontrolbad} is not optimal in all cases and 
    the key step of our subsequent analysis will make use of that difference. 
\end{rem}

\medskip

We will now focus on estimates that are valid up to and including the cores. 

By definition (see Appendix \ref{appendice:canon}), we have 
$$
r j(u_\xi^*) = -\nabla^\perp\big(r\Psi^*_\xi\big).
$$
Since the latter is singular at the points $a_i$ and not in $L^2_{\rm loc}$, there is no
hope that estimate \eqref{eq:dsd2} in Proposition \ref{prop:stronglocal} could
be extended to the whole of $\H.$ For that purpose, we have to replace $j(u_\xi^*)$
by some mollified version. The function $j(u^*_{\xi,\eps})$ would be a natural candidate,
but that would require that the vortex locations $\xi_i$ are known to a precision at
least as good as $\eps,$ which is not the case in view of \eqref{eq:dsd1}. For that reason,
instead we modify the function 
$\Psi_\xi^*$ to a function $\Psi_\xi^\natural$ 
in the following way (truncate $r\Psi_\xi^*$): 

We write  
\begin{equation}\label{eq:defreps}
    r_\xi := C_1 \eps \logeps^{C_1}e^{C_1\Sigma_a^r}
\end{equation}
and 
for each $i=1,\cdots,n$ we consider the connected 
component $\mathcal{C}_i$ of the superlevel set $\{ r\Psi_\xi^* \geq 
r\Psi^*_\xi(\xi_i+(r_\xi,0))\}$ (by convention we include $\xi_i$, where $\Psi^*_\xi$ is 
in principle not defined, in this set) 
which contains the point $\xi_i+(r_\xi,0)$, and we 
set $r\Psi^\natural_\xi = r\Psi^*_\xi(\xi_i+(r_\xi,0))$ inside $\mathcal{C}_i$. 
Next, we set $\Psi^\natural_\xi=\Psi^*_\xi$ on $\H 
\setminus \cup_{i=1}^n \mathcal{C}_i$ and finally we define
\begin{equation}\label{eq:defjnatural}
r j^\natural(u_\xi^*) = -\nabla^\perp\big(r\Psi^\natural_\xi\big).
\end{equation}
\begin{rem}\label{rem:coincide}
    Note that by construction $j^\natural(u_\xi^*)$ and $j(u_\xi^*)$ coincide everywhere
    outside $\cup_i \mathcal{C}_i$, that is everywhere except on a neighborhood of order 
    $r_\xi$ of the points $\xi_i,$ and that $j^\natural(u_\xi^*)\equiv 0$ inside each
    $\mathcal{C}_i.$ In the sense of distributions,
    \begin{equation}\label{eq:divjnatural}
    {\rm div}(rj^\natural(u_\xi^*)) = 0
    \end{equation}
    and
    \begin{equation}\label{eq:curljnatural}
        {\rm curl}(j^\natural(u_\xi^*)) =
        \sum_{i=1}^n|j(u_\xi^*)|d\mathcal{H}^1\rest_{\partial\mathcal{C}_i}.
    \end{equation}
\end{rem}

\begin{prop}\label{prop:closecore}
    In addition to the statements of Proposition \ref{prop:stronglocal}, there exists $\eps_2 \leq \eps_1$ 
    such that if $\eps \leq \eps_2$ then we have
\begin{equation}\label{eq:inthecore}
\int_{\H } \!\!  
r \left[ e_\eps(|u|) + 
\big|\frac{j(u)}{|u|}-j^\natural(u_\xi^*)\big|^2\right]  \leq 
C_2\big(\Sigma_a^r + \log\logeps\big),
\end{equation}
where $C_2$ depends only on $n$, $K_1$ and $r_0.$  
\end{prop}

The term $\log\logeps$ is not small and even diverging as $\eps\to 0,$ but since the main order 
for the energy in the core region is of size $\logeps$ that estimate will be sufficient for our 
needs. Away from the cores we will of course stick to estimate \eqref{eq:dsd2}.

\subsection{Time evolution of the Jacobian and conservation of momentum}\label{subsect:momentum}

For sufficiently regular solutions of \GP we have
\begin{equation}\begin{split}
    \partial_t (iv,\nabla v) &=(i\partial_t v,\nabla v) - (v,\nabla i \partial_t v)\\
                              &=(\frac{1}{r}{\rm div}(r\nabla v)+\frac{1}{\eps^2}v(1-|v|^2),\nabla v) -
    (v,\nabla(\frac{1}{r}{\rm div}(r\nabla v)+\frac{1}{\eps^2}v(1-|v|^2)))\\
                              &= \frac{2}{r}({\rm div}(r\nabla v),\nabla v) -
\nabla \left( \frac{1}{r}(v,{\rm div} (r\nabla v)) + \frac{1-|v|^4}{2\eps^2}\right).
\end{split}\end{equation}
Taking the {\it curl} of the previous identy and integrating against a test function 
$\varphi$ with bounded support and which vanishes at $r=0$ we obtain 
\begin{equation}\begin{split}\label{eq:fondam0}
    \frac{d}{dt}\int_\H  Jv\,  \varphi\, drdz 
    &= -\int_\H \eps_{ij} \frac{1}{r} (\partial_k(r\partial_k v),\partial_j v) \partial_i\varphi\\
    &= -\int_\H \eps_{ij} \frac{\partial_k r}{r} (\partial_j v  ,\partial_k v)\partial_i\varphi +\int_\H \eps_{ij}(\partial_jv,\partial_k v)\partial_{ik}\varphi\\
    & \qquad \ +\int_\H \eps_{ij}\partial_j(\frac{\sum_k|\partial_k v|^2}{2})\partial_i \varphi\\ 
    &= -\int_\H \eps_{ij} \frac{\partial_k r}{r} (\partial_j v  ,\partial_k v)\partial_i\varphi +\int_\H \eps_{ij}(\partial_jv,\partial_k v)\partial_{ik}\varphi\\
\end{split}\end{equation}
where we sum over repeated indices and since
$$
\int_\H \eps_{ij}\partial_j(\sum_k\frac{|\partial_k v|^2}{2})\partial_i \varphi= \int_\H \eps_{ij}(\sum_k\frac{|\partial_k v|^2}{2})\partial_{ij} \varphi = 0
$$
by anti-symmetry. In the sequel we will write
\begin{equation}\label{eq:defF}
\mathcal{F}(\nabla v, \varphi)
:= -\int_\H \eps_{ij} \frac{\partial_k r}{r} (\partial_j v  ,\partial_k v)\partial_i\varphi +\int_\H \eps_{ij}(\partial_jv,\partial_k v)\partial_{ik}\varphi,
\end{equation}
so that \eqref{eq:fondam0} is also rewritten as
\begin{equation}\label{eq:fondam1}
 \frac{d}{dt}\int_\H  Jv\,  \varphi\,  drdz = \mathcal{F}(\nabla v, \varphi),
\end{equation}
and is the equation from which the dynamical law for the vortex cores will be deduced.
For a real Lipschitz vector field $X=(X_r,X_z)$, we expand
\begin{equation}\label{eq:expandF}
    \mathcal{F}(X,\varphi)= \int_\H  -\frac{1}{r}X_rX_z\partial_r \varphi + \frac{1}{r} X_r^2 \partial_z \varphi 
                            + X_rX_z \partial_{rr}\varphi +X_z^2 \partial_{rz}\varphi 
                            -X_r^2\partial_{rz}\varphi - X_rX_z \partial_{zz}\varphi.
\end{equation}
Integrating by parts, we have
\begin{equation*}
    \begin{split}
        \int_\H X_rX_z\partial_{rr}\varphi &= \int_\H -\partial_r X_z X_r \partial_r \varphi - X_z\partial_r X_r \partial_r \varphi\\
                                           &=\int_\H (-\partial_zX_r - {\rm curl} X)X_r\partial_r \varphi + (\frac{1}{r}X_r+\partial_z X_z -\frac{1}{r}{\rm div}(rX))X_z\partial_r\varphi,\\
        \int_\H X_z^2\partial_{rz}\varphi &= \frac12 \int_\H X_z^2\partial_{rz}\varphi+\frac12 \int_\H X_z^2\partial_{rz}\varphi\\
                                          &= \int_\H -\frac12 \partial_z(X_z^2)\partial_r\varphi -\frac12 \partial_r(X_z^2)\partial_z \varphi,\\
        \int_\H -X_r^2 \partial_{rz}\varphi &= \int_\H \frac12 \partial_r(X_r^2)\partial_z\varphi + \frac12 \partial_z(X_r^2)\partial_r \varphi,
\end{split}
\end{equation*}
and
\begin{equation*}
    \int_\H -X_rX_z\partial_{zz}\varphi = \int_\H (\partial_rX_z-{\rm curl}X)X_z\partial_z \varphi + X_r(\frac{1}{r}{\rm div}(rX)-\frac{1}{r}X_r-\partial_rX_r)\partial_z\varphi,
\end{equation*}
so that after summation and simplification
\begin{equation}\label{eq:Fjoli}
    \mathcal{F}(X,\varphi)= \int_\H -({\rm curl}X)X\cdot \nabla \varphi +\frac{1}{r}{\rm div}(rX)X\times \nabla \varphi.
\end{equation}


Formally, the choice $\varphi = r^2$ in \eqref{eq:fondam1} leads 
to the conservation of the momentum along the z-axis 
$$
\frac{d}{dt}\int_\Omega  Jv\,  r^2 drdz = 0,
$$
but its justification would require additional arguments at infinity. In the next section we shall consider
a version of the momentum localized on some large but finite part of $\H.$

\subsection{Expansion of the main terms in the dynamics}

In this section we strengthen assumption $(H_1)$ into  
\begin{equation}\label{eq:H0}
 \begin{array}{c}
     \displaystyle \frac{K_0^{-1}}{\sqrt{\logeps}} \leq \min_{i\neq j} |a_{i}-a_{j}| 
     \leq \max_{i\neq j} |a_{i}-a_{j}| \leq \frac{K_0}{\sqrt{\logeps}}\\
\displaystyle \frac{r_0}{2} \leq \min_{i} r(a_{i})
\leq \max_{i} r(a_{i})\leq 2r_0\\
\end{array}\tag{${H_0}$}
\end{equation}
which is nothing but the time independent version of \eqref{eq:bonneechelle}, 
and we define $r_a$ and $\Sigma_a$ as in \eqref{eq:defra} and \eqref{eq:defSigma}. 
We shall also always implicitly assume that   
$$
\max_i |z(a_i)| \leq K_0.
$$
Since the problem is invariant under translation along the z-axis, and since we have
already assumed that all the points are close to each other (as expressed by the first
line in $(H_0)$), it is clear that this is not really an assumption but just a 
convenient way to avoid the necessity for various translations along the z-axis in some 
our subsequent claims.

Note that for sufficiently small $\eps$, and adapting the constant $K_0$ if necessary, the situation
described by $(H_0)$ indeed implies $(H_1)$, and therefore in the sequel we shall refer freely to 
the improved approximation points $\xi_i$ whose existence was established in Proposition \ref{prop:stronglocal}. 

Our analysis in the next sections will make rigorous the fact that the main contribution in the dynamical 
law for the vortex cores is obtained from \eqref{eq:fondam1}, with a suitable choice of test function $\varphi$,  
by replacing in the expression $\mathcal{F}(\nabla u,\varphi)$ the term $\nabla u$ by $j^\natural(u_\xi^*).$ 
Regarding $\varphi$, we assume that it satisfies  
\begin{itemize}\label{eq:assumevarphi}
    \item
        $\varphi$  is affine on each ball $B(\xi_i,\frac{1}{\logeps}),$
    \item
        $\varphi$ is compactly supported in the union of disjoint balls $\cup_i B(\xi_i,1/(2K_0\sqrt{\logeps})),$
    \item
        $|\nabla \varphi|\leq C$ and $|D^2\varphi| \leq CK_0\sqrt{\logeps}$, 
\end{itemize}
where $C$ is a universal constant for such a test function to exist. We will refer to the above 
requirement as condition $(H_\varphi).$ 


\begin{prop}\label{prop:maintermdynamics}
Under the assumptions $(H_0)$, \eqref{eq:defra}, \eqref{eq:boundinit} and $(H_\varphi)$, there exist $\eps_3 \leq \eps_2$ and $C_3$ depending
    only on $n$ and $K_0$ and $r_0$ such that if $\eps\leq \eps_3$ we have
$$
\left| \mathcal{F}(j^\natural(u_\xi^*),\varphi) - \sum_{i=1}^n 
\mathbb{J}\nabla_{a_i}H_\eps(\xi_1,\cdots,\xi_n)\cdot \nabla\varphi(\xi_i) \right| \leq C_3\left( 
    \Sigma_a^r + \log\logeps\right).
$$
\end{prop}

The main task in the remaining sections will be to control the discrepandcy between $\mathcal{F}(\nabla u,\varphi)$
and $\mathcal{F}(j^\natural(u_\xi^*),\varphi)$; for that purpose we will have to use the evolution equation 
to a larger extent (up to now our analysis was constrained on fixed time slices).   

\subsection{Approximation of the momentum}

As remarked earlier, the choice $\varphi=r^2$ in \eqref{eq:fondam1} formally leads to the conservation of the 
momentum $\int Ju r^2 \, drdz.$ Yet, giving a clear meaning to the previous integral and proving its conservation
in time is presumably not an easy task. Instead, we will localise the function $r^2$ by cutting-it off 
sufficiently far away from the origin and derive an approximate conservation law. More precisely,
we set 
$$
R_\eps := \logeps^2
$$
and we let $0\leq \chi_\eps\ \leq 1$ be a smooth cut-off function with compact support in $[0,2R_\eps]\times[-2R_\eps,2R_\eps]$ 
and such that $\chi_\eps \equiv 1$ on $[0,R_\eps]\times[-R_\eps,R_\eps]$ and $|\nabla \chi_\eps| \leq C/R_\eps.$
In the sequel we write
\begin{equation}\label{eq:defPeps}
    P_\eps(u) := \int_\H Ju r^2 \chi_\eps \, drdz. 
\end{equation}

\begin{prop}\label{prop:momentumapprox}
    Under the assumption $(H_0)$ and \eqref{eq:boundinit}, there exist $\eps_4 \leq \eps_3$ such that
if $\eps \leq \eps_4$ then we have :
\begin{equation}\label{eq:approxJ}
    \left|P_\eps(u)  - P(\xi_1,\cdots,\xi_n)\right| \leq C_4 \frac{(1+\Sigma_\xi)^2}{\logeps^2},
\end{equation}
and
\begin{equation}\label{eq:approxddtJ}
    \left|  \partial_t P_\eps(u)\right| \leq C_4 \frac{1+\Sigma_\xi}{R_\eps} = C_4\frac{1+\Sigma_\xi}{\logeps^2}, 
\end{equation}
where $C_4$ depends only on $n$, $K_0$ and $r_0.$   
\end{prop}

\subsection{A key argument}

Coming back to Remark \ref{rem:1} and Remark \ref{rem:crucial} we now state 

\begin{prop}\label{cor:apprixgood}
Under the assumptions $(H_0)$ and \eqref{eq:boundinit}, there exists $\eps_5
\leq \eps_4$ and $\sigma_5>0$,
depending only on $K_0$ and $n$, such that if $\eps \leq \eps_5$ and if
\begin{equation}\label{eq:sigma5}
 \Sigma_a + |H_\eps(a_1,\cdots,a_n)-H_\eps(\xi,\cdots,\xi_n)| \leq \sigma_5 \logeps,
 \end{equation}
then
\begin{equation}\label{eq:sigmaxicontrol}
    \Sigma_\xi \leq 2\Sigma_a + C_5\left[ r_a\sqrt{\logeps} +\frac{1}{\logeps}+ \logeps\,|P_\eps(u)-P(a_1,\cdots,a_n)|\right] 
\end{equation}
where $C_5$ depends only on $n$, $K_0$ and $r_0.$  
\end{prop}
\begin{proof}
    For a quantity $f$ we temporarily write $\Delta f := |f(a_1,\cdots,a_n)-f(\xi_1,\cdots,\xi_n)|$ when the latter has a 
    well defined 
    meaning.  By the triangle inequality we have 
    $$
    \Delta H_\eps \leq \Delta \left( H_\eps-\frac{\logeps}{2r_0}P\right)+ \frac{\logeps}{2r_0} \Delta P,
    $$
    and also
    $$
    \Delta P \leq \left| P_\eps(u)-P(\xi_1,\cdots,\xi_n)\right| + \left|P_\eps(u)-P(a_1,\cdots,a_n)\right|.
    $$
    In view of the expansion in Remark \ref{rem:1} (the $o(1)$ holds in particular in $\mathcal{C}^1$ norm under
    assumption $(H_0)$), we have
    $$
    \Delta \left( H_\eps-\frac{\logeps}{2r_0}P\right) \leq
    C|(\xi_1-a_1,\cdots,\xi_n-a_n)|\sqrt{\logeps}
    $$
    and by \eqref{eq:defra} and \eqref{eq:dsd1}
    $$
    |(\xi_1-a_1,\cdots,\xi_n-a_n)|\leq C (r_a+ \eps\logeps^{C_1}e^{C_1\Sigma_a^r}). 
    $$
    By \eqref{eq:approxJ} we also have
    $$
    \frac{\logeps}{2r_0}\left| P_\eps(u)-P(\xi_1,\cdots,\xi_n)\right|\leq
    \frac{C_4}{2r_0}
    \frac{(1+\Sigma_\xi)^2}{\logeps} \leq  \frac{C_4}{2r_0} \frac{(1+\Sigma_a + \Delta
    H_\eps)^2}{\logeps}  
    $$
    and 
    $$
    \frac{C_4}{2r_0} \frac{(1+\Sigma_a + \Delta 
    H_\eps)^2}{\logeps} \leq     \frac{3C_4}{2r_0} \frac{1+\Sigma_a^2 +
        (\Delta
    H_\eps)^2}{\logeps} \leq \frac{3C_4}{2r_0}\left( \frac{1}{\logeps} +
    \sigma_5  (\Sigma_a+\Delta H_\eps)\right).
    $$
    By summation of all the inequalities gathered so far we obtain
    \begin{equation}\begin{split}
        \Delta H_\eps &\leq C\left((r_a+\eps\logeps^{C_1}e^{C_1\Sigma_a^r})\sqrt{\logeps} +
    \frac{1}{\logeps} + \Sigma_a + \logeps\,
    |P_\eps(u)-P(a_1,\cdots,a_n)|\right)\\
    & \qquad +  \frac{3C_4}{2r_0}\sigma_5 \Delta H_\eps.  
\end{split}\end{equation}
    We therefore choose $\sigma_5$ in such a way that $\frac{3C_4}{2r_0}\sigma_5 \leq
    \frac12$, and we may then absorb the last term of the previous
    inequality in its left-hand side. Combined with the fact that $\Sigma_\xi 
    \leq \Sigma_a+\Delta H_\eps$ the conclusion \eqref{eq:sigmaxicontrol} follows.  
\end{proof}

\begin{rem}
    The main gain in \eqref{eq:sigmaxicontrol} is related to the fact that in
    the right-hand side we have a term of the form $r_a\sqrt{\logeps}$ rather
    than (the easier) $r_a\logeps$ which would have followed from a crude
    gradient bound on $H_\eps.$ Note however that we have exploited here the 
    assumption $(H_0)$, that is the fact that all the cores are of order
    $\sqrt{\logeps}$ apart from each other, whereas Proposition
    \ref{prop:stronglocal} holds under the weaker assumption $(H_1)$.

The right-hand side of \eqref{eq:sigmaxicontrol} also contains a term 
involving $P_\eps(u)$ and $P(a_1,\cdots,a_n).$ When introducing time dependence
in the next sections, we will take advantage of the fact that $P$ is preserved
by the ODE flow $\lf_\eps$ and that $P_\eps$ is almost preserved by the PDE flow 
$\gp$, as already expressed in \eqref{eq:approxddtJ}.
\end{rem}


\subsection{Time dependence and Stopping time}

In this section we introduce time dependence and go back to the setting of Theorem \ref{thm:main}, 
that is we assume \eqref{eq:bonneechelle}  and \eqref{eq:excesspetit}.
For $s \in [0,S_0],$ we define the localization scales
\begin{equation}\label{eq:locals}
 r_a^s := \Big\|J u_\eps^s - \pi \sum_{i=1}^n 
 \delta_{a_{i,\eps}(s)}\Big\|_{\dot W^{-1,1}(\Omega_0)},
\end{equation}
and the excess energy 
\begin{equation}\label{eq:excs}
\Sigma^s := \left[\Ew(u_\eps^s) - 
H_{\eps}(a_{1,\eps}(s),\cdots,a_{n,\eps}(s))\right]^+
\end{equation}
where we recall that $\Omega_0= \{r\geq \frac{r_0}{4}\}$ and $u_\eps^s$ is the solution
of $\gp$ evaluated at time $t=s/\logeps.$

Since $\Ew$ is preserved by the flow of $\gp$ and since $H_\eps$ is preserved by $\lf_\eps,$ 
we have
\begin{equation}\label{eq:excesscons}
    \Sigma^s = \Sigma^0 \qquad \forall s \in [0,S_0].
\end{equation}
We introduce the stopping time
\begin{equation}\label{eq:stop}
    S_{\rm stop} := \inf\left\{S \in [0,S_0],\ r_a^s \leq \frac{\rho_{\rm min}}{8},  \quad \forall s \in [0,S]\right\},  
\end{equation}
where we have set, in view of \eqref{eq:bonneechelle},  
\begin{equation}\label{eq:defrhomin}
    \rho_{\rm min} := \frac{K_0^{-1}}{\sqrt{\logeps}}.
\end{equation}
By \eqref{eq:excesspetit} and continuity it is clear that $S_{\rm stop}>0,$ at least provided $\eps_0$ and $\sigma_0$
are chosen small enough. By construction, we also have
$$
r_a^s < \frac{\rho_{a_\eps(s)}}{4} \qquad \forall s \in [0,S_{\rm stop}],
$$
and likewise by \eqref{eq:excesscons}  
$$
\Sigma^s + r_a^s\logeps \leq \sigma_1 \logeps,
$$
where $\sigma_1$ is given in Proposition \ref{prop:stronglocal}.
Applying Proposition \ref{prop:stronglocal} for each $s \in [0,S_{\rm stop}]$, we get 
    functions $s\mapsto \xi_i(s) \equiv \xi_i^s$, $i=1,\cdots,n.$ By continuity of the flow map for $\gp$, 
and doubling $C_1$ if necessary, we may further assume that these maps are piecewise constant
and hence measurable on $[0,S_{\rm stop}].$ In the sequel, in view of \eqref{eq:dsd1}, we set
    \begin{equation}\label{eq:defrxis}
        r_\xi^s = C_1\eps\logeps^{C_1}e^{C_1(\Sigma^0+r_a^s\logeps)} \stackrel{\eqref{eq:securitybounds}}{\leq} \eps^\frac56, 
    \end{equation}
for each $s \in [0,S_{\rm stop}].$

\medskip
The following Proposition yields a first estimate on the time evolution of the vortex cores. At this stage it
does not contain any information about the actual motion law, but only a rough (but essential) Lipschitz bound.

\begin{prop}\label{prop:firstspeedbound}
    For sufficiently small values of $\sigma_0$ and $\eps_0,$ whose threshold may be chosen depending only on 
    $n$, $K_0$ and $r_0,$ the following holds:  
    There exist $C_6>0,$ also depending only on $n$, $K_0$ and $r_0,$ such that for all $s_1,s_2 \in [0,S_{\rm stop}]$ 
    such that $s_1 \leq s_2 \leq s_1 + {\logeps}^{-1}$ we have  
    \begin{eqnarray}\label{eq:speedbound}
        &\|Ju_\eps^{s_1} - Ju_\eps^{s_2}\|_{\dot W^{-1,1}(\Omega_0)} \leq C_6(|s_1-s_2|+r_\xi^{s_1}),\\
        &r_\xi^{s_2} \leq r_\xi^{s_1}\left(1+ C_6\left(|s_2-s_1|+\eps^\frac56\right)\logeps\right),\label{eq:rxiequiv}\\
        &\left\{ a_{i,\eps}(s_2),\xi_i(s_2)\right\} \subset B(a_{i,\eps}(s_1),\frac{\rho_{\rm min}}{4})\label{eq:dansboule}.
    \end{eqnarray}
    Moreover, if $r_a^{s_1} \leq \rho_{\rm min}/16,$ then $S_{\rm stop} \geq s_1+ (C_6\sqrt{\logeps})^{-1}.$
   \end{prop}

\subsection{Control of the discrepancy}

The following proposition is the final ingredient leading to the proof of Theorem \ref{thm:main}, it
can be regarded as a discrete version of the Gronwall inequality for the quantity $r_a^s.$ 

\begin{prop}\label{prop:controldiscrep}
Assume that $s< S_{\rm stop}$ and that $r_a^s \leq \rho_{\min}/16$ and set
$$S := s + \frac{(r_\xi^s)^2}{\eps}.$$
Then $S < S_{\rm stop}$ and
\begin{equation}\label{eq:discreteg}
\frac{r_a^S-r_a^s}{S-s} \leq C_0\left(r_a^s + \frac{\Sigma^0+r_a^0\logeps}{\sqrt{\logeps}} + 
\frac{C_\delta}{\logeps^{1-\delta}}\right),
\end{equation}
where $C_0$ depends only on $n$, $K_0$ and $r_0$, $\delta >0$ is arbitrary and $C_\delta$ 
depends only on $\delta.$ 
\end{prop}

\begin{rem}
    The time step $S-s$ on which the differential inequality \eqref{eq:discreteg} holds is not arbitrary,
    in view of \eqref{eq:defrxis} it satisfies
    $$
    S-s = C_1^2 \eps \logeps^{2C_1}e^{2C_1(\Sigma^0+r_a^s\logeps)}
    $$
    which, for $\eps$ sufficiently small, is both large with respect to $\eps$ and small with respect to
    lower powers of $\eps.$ The fact that it is large with respect to $\eps$, as the proof of Proposition 
    \ref{prop:controldiscrep} will show, is essential in order to allow the averaging effects of the continuity
    equation (see \eqref{eq:continuity}) to act. On the other hand, the fact that it is small with respect
    to lower powers of $\eps$ will allow us, when using it iteratively, to rely on the softer estimates of Proposition \ref{prop:firstspeedbound}
    to bridge the gaps between the discrete set of times so obtained and the full time interval $[0,S_0]$ 
    which appears in the statement of Theorem \ref{thm:main}.
\end{rem}

\section{Proofs}

\noindent{\bf Proof of Lemma \ref{lem:energy}}
It suffices to combine the expansion of Lemma \eqref{lem:hors_disque0} with those (see e.g. \cite{BBH})
for the optimal Ginzburg-Landau profile $f_\eps.$\qed

\bigskip

\noindent{\bf Proof of Proposition \ref{prop:stronglocal}}
We divide the proof in several steps. 
We first set 
$$
\lea = 4\max(\frac{1}{\logeps},r_a).
$$

{\noindent \bf Step 1 : rough lower energy bounds on 
\mb$B(a_i,\lea)$.\mn}
In view of our assumptions and the fact that the $\dot W^{-1,1}$ is decreasing with respect
to the domain, we are in position, provided $\eps_1$ and $\sigma_1$ are sufficiently small (depending only
on $r_0$ and $K_1$), to apply Theorem \ref{thm:B.1} after translation to the balls $B(a_i,\lea).$ This yields the 
lower bounds
\begin{equation}\label{eq:lowerbloc}
    \begin{split}
        \E(u_\eps^s,B(a_i,\lea)) &\geq \pi \log\frac{\lea}{\eps} + \gamma 
    - \frac{C}{\lea} \left(\eps \sqrt{\log(\lea/\eps)} + r_a\right)\\
    & \geq \pi \logeps -C \left(r_a\logeps + \log\logeps\right),
\end{split}
\end{equation}
for any $1\leq i \leq n,$ where $C$ is universal provided we require that 
$\eps_1$  is also  
sufficiently small so that $\log\logeps\geq 1$ for $\eps\leq \eps_1.$ From 
\eqref{eq:lowerbloc} and the global energy bound given by the assumption of $\Sigma_a$ it 
follows, comparing the weight function $r$ with its value $r(a_i)$, that
\begin{equation}\label{eq:lowerblocw}
 \Ew(u,B(a_i,\lea)) \geq \pi r(a_i) 
\logeps  - C \left( r_a\logeps + \log\logeps 
\right)
\end{equation}
for any $1\leq i \leq n,$ and for a possibly larger constant  
$C$ depending only on $K_1$, $r_0$ and $n.$

{\noindent \bf Step 2 : rough upper energy bounds on \mb $\H\setminus 
\cup_{i=1}^n B(a_i,\lea/2)$ and  $B(a_i,2\lea)$.\mn}
The equivalent of \eqref{eq:lowerblocw} with $\lea$ replaced by 
$\lea/2$, combined with the global upper energy bound given by the 
definition of $\Sigma_a$, yields the upper bound
\begin{equation}\label{eq:upperout}
\Ew(u,\H\setminus \cup_{i=1}^n B(a_i,\frac{\lea}{2}) \leq C \left( 
 \Sigma_a^r + \log\logeps\right),
\end{equation}
where $C$ depends only $K_0$ and $n$. Also, combining \eqref{eq:lowerblocw} 
(for all 
but one $i$) with the definition of $\Sigma_a$, we obtain the upper bound
\begin{equation*}\label{eq:upperinw}
\Ew(u,B(a_i,2\lea)) \leq \pi r(a_i)\logeps + C \left( 
\Sigma_a^r + \log\logeps\right),
\end{equation*}
for any $1\leq i \leq n,$ and therefore
\begin{equation}\label{eq:upperin}
\E(u,B(a_i,2\lea)) \leq \pi 
\log\frac{2\lea}{\eps}  + C \left(  
\Sigma_a^r + \log\logeps\right),
\end{equation}
for any $1\leq i \leq n,$ where $C$ depends only on $K_0$ and $n.$ 

{\noindent \bf Step 3 : first localisation estimates.}
We apply Theorem \ref{thm:B.2}, after translation, to each of the balls 
$B(a_i,2\lea)$, and we denote by $\xi_i$ the corresponding points. In 
view of \eqref{eq:upperin}, this yields 
\begin{equation}\label{eq:concentrinmoyen}
\sum_{i=1}^n\|Ju-\pi\delta_{\xi_i}\|_{\dot W^{-1,1}(B(a_i,2\lea))} \leq  
\eps e^{C(\Sigma_a^r+\log\logeps)}\leq \eps\logeps^C e^{C\Sigma_a^r}.
\end{equation}
Note that from \eqref{eq:concentrinmoyen} and the definition of $r_a$ in \eqref{eq:defra}
we have the bound
\begin{equation}\label{eq:aprochexi}
    \max_{i=1,\cdots,n} | a_i-\xi_i| \leq \frac{1}{\pi}\left( r_a + \eps\logeps^C e^{C\Sigma_a^r}\right).
\end{equation}
Provided $\eps_1$ and $\sigma_1$ are sufficiently small, this also implies that
$$
B(\xi_i,\le) \subset B(a_i,\lea) \qquad \forall i=1\cdots,n,
$$
where we have set
$$
\le := \frac{1}{\logeps}.
$$
From now on we will rely entirely on the points $\xi_i$ rather than on the $a_i$ for our constructions.

{\noindent \bf Step 4 : improved lower energy bounds close to the cores.}
We apply Theorem \ref{thm:B.1}, after translation, to 
each of the balls $B(\xi_i,\rho),$ where $\le/2 \geq \rho 
\geq \eps^\frac45$ is some free parameter which we will fix later. Since $\dot 
W^{-1,1}$ norms are monotone functions of the domain and since  
$B(\xi_i,\rho)\subset B(a_i,\lea)$ by \eqref{eq:aprochexi}, in view of \eqref{eq:concentrinmoyen} 
we obtain
\begin{equation}\label{eq:lowerblocb}
 \E(u,B(\xi_{i},\rho)) \geq \pi \log\frac{\rho}{\eps} + \gamma 
- C\eps\logeps^Ce^{C\Sigma_a^r}\rho^{-1},
\end{equation}
and therefore
\begin{equation}\label{eq:lowerblocwb}
 \Ew(u,B(\xi_{i},\rho )) \geq r(\xi_{i})\Big(\pi 
\log\frac{\rho}{\eps} + \gamma\Big) - C 
(\eps\logeps^Ce^{C\Sigma_a^r}\rho^{-1}+\rho \logeps),
\end{equation}
for $i=1,\cdots,n.$

Note that taking $\rho=\le/2$ and then arguing exactly as in Step 2 yields 
the slight variant of \eqref{eq:upperout}:
\begin{equation}\label{eq:upperoutbis}
\Ew(u,\H\setminus \cup_{i=1}^n B(\xi_i,\frac{\le}{2}) \leq C \left( 
 \Sigma_a^r + \log\logeps\right).
\end{equation}
Yet at this point we wish to keep $\rho$ as a free parameter. 

{\noindent \bf Step 5 : towards lower energy bounds away from the cores.} 
In this step we compare $u$, away from 
the cores, with the singular vortex ring $u_\xi^*$. For 
convenience, we simply denote $j(u^*_\xi)$ by $j^*$, we
let $0 \leq \chi \leq 1$ be a lipschitz function on  
$\H$, and we set 
\begin{equation}\label{eq:defHomegaxi}
    \H_{\xi,\rho} :=
\H \setminus \cup_{i=1}^n B(\xi_i,\rho).
\end{equation}
\noindent
The starting point is the pointwise equality
\begin{equation}\label{eq:funddecomp}
 e_\eps(u) = \frac12 |j_*|^2 +j_*\big(\frac{j(u)}{|u|}-j_*\big) +  
e_\eps(|u|) + \frac12 \big| \frac{j(u)}{|u|}-j_*\big|^2,
\end{equation}
which holds almost everywhere in $\H.$ Notice that 
all the terms in the right-hand side of \eqref{eq:funddecomp} are pointwise 
non-negative except possibly the second one. We integrate \eqref{eq:funddecomp} 
multiplied by $\chi^2$ on $\H_{\xi,\rho}$ and estimate the 
corresponding terms. 

We first write
$$
\int_{\H_{\xi,\rho}}  r j_*\big(\frac{j(u)}{|u|}-j_*\big)\chi^2 = 
\int_{\H_{\xi,\rho}} r j_*\big(j(u)-j_*\big)\chi^2 + 
\int_{\H_{\xi,\rho}} r j_*\big(\frac{j(u)}{|u|}-j(u)\big)\chi^2
$$
and we readily estimate
$$
\Big| \int_{\H_{\xi,\rho}} r j_*\big(\frac{j(u)}{|u|}-j(u)\big)\chi^2 \Big| 
\leq 
\frac{C}{\rho} \Big( \int_{\H_{\xi,\rho}}  r 
\frac{j^2(u)}{|u|^2}\Big)^\frac12 
\Big( \int_{\H_{\xi,\rho}} r 
(1-|u|)^2\Big)^\frac12 \leq C\frac{\eps}{\rho}\logeps,
$$
where we have used the facts that $|j^*| \leq C/\rho$ on $\H_{\xi,\rho}$ and that the last two integral factors  
are dominated by (a constant multiple of) the weighted energy. 
By definition (see 
Appendix \ref{appendice:canon}), we have 
$$
r j_* = -\nabla^\perp\big(r\Psi^*_\xi\big).
$$
We modify (truncate) the function 
$\Psi_\xi^*$ to a function $\tilde \Psi_\xi^*$ 
in the following way : for each $i=1,\cdots,n$ we consider the connected 
component $\mathcal{C}_i$ of the superlevel set $\{ \Psi_\xi^* \geq 
\Psi^*_\xi(\xi_i+(\rho,0))\}$ (by convention we include $\xi_i$, where $\Psi^*_\xi$ is 
in principle not defined, in this set) 
which contains the point $\xi_i+(\rho,0)$, and we 
set $\tilde \Psi^*_\xi = \Psi^*_\xi(\xi_i+(\rho,0))$ on $\mathcal{C}_i$. 
Next, we set $\tilde \Psi^*_\xi=\Psi^*_\xi$ on $\H 
\setminus \cup_{i=1}^n \mathcal{C}_i.$ By construction,
$$
-\nabla^\perp\big(r\tilde \Psi^*_\xi\big)= 
rj_*\mathbbm{1}_{\H \setminus \cup_{i=1}^n 
\mathcal{C}_i}, 
$$
so that 
$$
-\nabla^\perp\big(r\tilde \Psi^*_\xi\chi^2\big)= 
rj_*\mathbbm{1}_{\H_{\xi,\rho}} \chi^2 + 
rj_*\big[\sum_{i=1}^n(\mathbbm{1}_{B(\xi_i,\rho)}-\mathbbm{1}_{\mathcal{C}_i}
)\big]\chi^2 - 2r 
\tilde \Psi_\xi^* \chi \nabla^\perp\chi.
$$
The latter and integration by parts yields 
\begin{equation}\label{eq:tirole}
\begin{split}
\Big|\int_{\H_{\xi,\rho}} r j_*\big(j(u)-j_*\big)\chi^2 \Big| &\leq  
\sum_{i=1}^n\int_{B(\xi_i,\rho)\bigtriangleup \mathcal{C}_i} 
r|j^*|\big|j(u)-j_*\big|
+ \int_\H 2r 
|\tilde \Psi_\xi^*| |\nabla^\perp\chi|\chi|j(u)-j_*|\\
& \quad + \Big|\int_{\H}2
r\tilde \Psi^*_\xi \big(J(u)-J(u^*_\xi)\big)\chi^2\Big| .
\end{split}
\end{equation}
In order to bound the right-hand side of \eqref{eq:tirole} we first remark
that, from \eqref{eq:Aa} and \eqref{eq:dAa} in the Appendix, for each 
$i=1,\cdots,n,$  we have
\begin{equation}\label{eq:bornesurf}
    d_{\mathcal{H}}( \mathcal{C}_i, B(\xi_i,\rho)) \leq  
\frac{\rho^2}{\rho_a}\log(\frac{\rho}{\rho_a}), \quad\text{and hence}\quad 
\mathcal{L}^2\big(\mathcal{C}_i \bigtriangleup B(\xi_i,\rho)\big) \leq  
C \frac{\rho^3}{\rho_a}\log(\frac{\rho}{\rho_a}).
\end{equation}
We write
$$
\int_{B(\xi_i,\rho)\bigtriangleup \mathcal{C}_i} 
r|j^*|\big|j(u)-j_*\big| \leq \int_{B(\xi_i,\rho)\bigtriangleup \mathcal{C}_i} 
r|j^*|\big|\frac{j(u)}{|u|}-j_*\big| + \int_{B(\xi_i,\rho)\bigtriangleup \mathcal{C}_i} 
r\eps|j^*|\big|\frac{j(u)}{|u|}\big|\frac{||u|-1|}{\eps},  
$$
and since $|j^*|\leq C/\rho$ on $\mathcal{C}_i \bigtriangleup 
B(\xi_i,\rho)$, we deduce from \eqref{eq:bornesurf}, the Cauchy-Schwarz inequality, and global energy upper 
bounds, that  
\begin{equation}\label{eq:conjg1}
\sum_{i=1}^n\int_{B(\xi_i,\rho)\bigtriangleup \mathcal{C}_i} 
r|j^*|\big|j(u)-j_*\big| \leq 
C \left( \frac{\rho}{\rho_a}\log(\frac{\rho}{\rho_a})\logeps \right)^\frac12 + C \frac{\eps}{\rho}\logeps.
\end{equation}
Concerning the second error term in \eqref{eq:tirole}, we first decompose it as
\begin{equation}\label{eq:conjg2}
\int_\H r 
|\tilde \Psi_\xi^*| |\nabla^\perp\chi| |j(u)-j_*| \chi=\int_\H r 
|\tilde \Psi_\xi^*| |\nabla^\perp\chi| |\frac{j(u)}{|u|}-j_*|\chi +\int_\H r 
|\tilde \Psi_\xi^*| |\nabla^\perp\chi| |\frac{j(u)}{|u|}||1-|u||\chi 
\end{equation}
and we write by Cauchy-Schwarz inequality on one hand
\begin{equation}\label{eq:conjg3}
\int_\H r 
|\tilde \Psi_\xi^*| |\nabla^\perp\chi||\frac{j(u)}{|u|}-j_*|\chi  \leq C\|\nabla \chi\|_\infty 
\Big(\int_{{\rm spt}(\nabla \chi)} r |\tilde \Psi^*_\xi|^2\Big)^\frac12 
\Big(\int_{{\rm spt}(\nabla \chi)} r (e_\eps(u)+e_\eps(u^*_\xi))\chi^2 \Big)^\frac12,
\end{equation}
and by direct comparison with the energy density on the other hand 
\begin{equation}\label{eq:conjg4}
\int_\H r 
|\tilde \Psi_\xi^*| |\nabla^\perp\chi| |\frac{j(u)}{|u|}||1-|u||\chi \leq C \eps \logeps^2  \|\nabla \chi\|_\infty. 
\end{equation}
Coming back to \eqref{eq:funddecomp}, and taking into account \eqref{eq:conjg1}-\eqref{eq:conjg4}, we conclude that
\begin{equation}\label{eq:premestim}
 \int_{\H_{\xi,\rho}} \!\!\!\!\!r e_\eps(u) \chi^2  \geq 
 \int_{\H_{\xi,\rho}}\!\!\!\!\!  r \left[ \frac{|j_*|^2}{2} + e_\eps(|u|) + 
\big|\frac{j(u)}{|u|}-j_*\big|^2\right] \chi^2 - \left|\int_{\H}2
r\tilde \Psi^*_\xi \big(J(u)-J(u^*_\xi)\big)\chi^2\right|- {\rm Err}(\chi^2),
\end{equation}
where
\begin{equation}\begin{split}\label{eq:Errexpand}
    {\rm Err}(\chi^2) \leq& 
    C \Big[ \Big( \frac{\rho}{\rho_a}\log(\frac{\rho}{\rho_a})\logeps \Big)^\frac12 + \frac{\eps}{\rho}\logeps \\
&+ \|\nabla \chi\|_\infty 
\Big(\int_{{\rm spt}(\nabla \chi)} r |\tilde \Psi^*_\xi|^2\Big)^\frac12 
\Big(\int_{{\rm spt}(\nabla \chi)} r (e_\eps(u)+e_\eps(u^*_\xi))\chi^2 \Big)^\frac12\\
&+  \|\nabla \chi\|_\infty\eps \logeps^2 \Big].
\end{split}\end{equation}

{\noindent \bf  Step 6 : improved lower energy bounds away from the 
cores.}
The right-hand side of estimate \eqref{eq:premestim} contains quantities which we do not yet control:    
 we need good localisation estimates for the jacobian $Ju$, also outside the cores, and we also have to get
 rid of the energy term due to the cut-off in \eqref{eq:Errexpand}. 
To deal with the localisation, we shall rely on Theorem \ref{thm:B.1bis}, but in view of 
the difference between $\E$ and $\Ew$ (the factor $r$), we only expect good 
localisation estimates when $r$ is not too small. To quantify this, we define 
the set
$$
\mathcal{S}= \left\{ s=2^{-k}, k \in \Z, \text{ s.t. } \Ew\Big(u,\{s\leq r \leq 2s\} 
\setminus \cup_{i=1}^nB(\xi_i,\frac{\le}{2})\Big) \leq \frac{\pi}{12} 
s\logeps \right\} 
$$
and the value  
$$
r_\mathcal{S} := \min \left\{s=2^{-k}, k \in \mathbb{Z}, \text{ s.t. } 2^{-\ell} 
    \in \mathcal{S} \ \forall \ell \leq k
    \right\}.
$$
Note that by \eqref{eq:upperoutbis} we have
\begin{equation}\label{eq:sauveur}
    r_\mathcal{S} \leq C \left( \frac{\Sigma_a^r}{\logeps} + 
    \frac{\log\logeps}{\logeps}\right),
\end{equation}
which we will improve later on in \eqref{eq:bieneu}.
Also, whenever $\Omega$ is an open bounded subset contained in $\{s\leq r \leq 
2s\} \setminus \cup_{i=1}^nB(\xi_i,\frac{\le}{2})$ for some $s\geq 
r_\mathcal{S}$,  covering it with two of the above slices we obtain
$$
\E(u,\Omega)\leq \frac{1}{s}\Ew(u,\Omega) \leq \frac{\pi}{4} \logeps
$$
and therefore by Theorem \ref{thm:B.1bis}
\begin{equation}\label{eq:concentroutmoyen}
 \|Ju\|_{\dot W^{-1,1}(\Omega)} \leq  
C\E(u,\Omega)\eps^\frac34.
\end{equation}

We now take $\rho:=\re=\eps^\frac23$, and in view of \eqref{eq:sauveur} we let 
$ r_\mathcal{S} \leq \tilde r \leq C \left( \Sigma_a^r + 
    \log\logeps\right)/\logeps.$  We choose $0\leq \chi \leq 1$ a 
lipschitz function supported in $\{r\geq \tilde r\}$ and such that $\chi \equiv 1$ on 
$\{r\geq 2\tilde r\}$ and $|\nabla \chi| \leq C/\tilde r.$ We then invoke  
estimate \eqref{eq:premestim} of Step 5, which we add-up with estimate 
\eqref{eq:lowerblocwb} (note that $\chi\equiv 1$ on each $B(\xi_i,\rho)$ by definition of $\tilde r$, at 
least provided $\eps_1$ and $\sigma_1$ are chosen small enough) to write
\begin{equation}\label{eq:norway}
\Ew(u,\{r\geq \tilde r\}) \geq \int_{\H} re_\eps(u)\chi^2 \geq  T_1-T_2-T_3+T_4
\end{equation}
where
\begin{equation*}
T_1 =  \int_{\H_{\xi,\rho}}\!\!\!\!\!  r  
\frac{|j_*|^2}{2} + \sum_{i=1}^n r(\xi_{i})\Big(\pi \log\frac{\re}{\eps} + \gamma\Big) 
- 
\int_{\{r\leq \tilde r\}} r  
\frac{|j_*|^2}{2}, 
\end{equation*}
$$
T_2 =  {\rm Err}(\chi^2) \qquad \text{,}\qquad T_3:= \Big|\int_{\H}2
r\tilde \Psi^*_\xi \big(J(u)-J(u^*_\xi)\big)\chi^2\Big|,
$$
and
$$
T_4 =  \int_{\H_{\xi,\rho}}\!\!\!\!\!  r \left[  e_\eps(|u|) + 
\big|\frac{j(u)}{|u|}-j_*\big|^2\right] \chi^2 \geq 0. 
$$
We invoke Lemma \ref{lem:hors_disque0} (with $a=\xi$) and the definition 
of $H_\eps$  to obtain
\begin{equation}\label{eq:tarik1}
    T_1 \geq  H_\eps(\xi) -C\big( \tilde r^2 + \eps^\frac23\logeps^3\big), 
\end{equation}
where we have also used $(H_1)$ in order to get rid of $\rho_a$ wherever
it appeared. Invoking \eqref{eq:Aaz} to compute some of the terms in \eqref{eq:Errexpand}, we also obtain
$$T_2 \leq 
C\left(
\eps^\frac13\logeps^\frac32 + \tilde r\Big( \Ew(u,\{\tilde r\leq r 
\leq 2\tilde r\}) + \tilde r\Big)^\frac12 + \frac{\eps}{\tilde r}\logeps^2\right),
$$
and since $\tilde r \geq r_\mathcal{S}$ the definition of the latter yields
\begin{equation}\label{eq:tarik2}
T_2 \leq 
C\left(
\eps^\frac13\logeps^\frac32 + \tilde r^\frac32 \logeps^\frac12 + \frac{\eps}{\tilde r}\logeps^2\right).
\end{equation}
It remains to estimate $T_3$ for which we will rely on \eqref{eq:concentrinmoyen} and \eqref{eq:concentroutmoyen}. 
To that purpose, we write
$$
\chi^2 = \chi^2\left( \sum_{i=1}^n \psi^{\rm in}_i + \sum_{j \in \N }
\psi_j^{\rm out}\right) \qquad \text{on } \H
$$
for an appropriate partition of unity on $\H$ verifying the 
following : 
\begin{enumerate}
 \item 
 Each function of the partition is $\mathcal{C}^\infty$ smooth and compactly 
supported, its support has a smooth boundary.
\item
 Each point of $\H$ is contained in the support of at most four functions of 
the partition.
\item
We have 
 $$
{\rm spt}(\psi^{\rm in}_i) \subset B(\xi_i,\le),\qquad |\nabla   
\psi^{\rm 
in}_i|\leq C/\le, \qquad \forall i=1,\cdots,n,
$$
$$
{\rm spt}(\psi^{\rm out}_j) \subset \H \setminus \cup_{i=1}^n 
B(\xi_i,\le/2), \qquad \forall j \in \N.
$$
\item[]
For each $j \in \N$, there exists $r_j>\tilde r/2$ such that
$$
{\rm spt}(\psi^{\rm out}_j) \subset \{r_j \leq r \leq 2r_j\} 
\qquad\text{and}\qquad |\nabla   \psi^{\rm 
out}_j|\leq C\left( \frac{1}{r_j} + \frac{1}{\le}\right).
$$    
\end{enumerate}
The existence of such a partition can be obtained by covering $\H$ 
with rectangular tiles with a step size close to being dyadic in the $r$ direction and 
constant in the $z$ direction and then arranging the round holes corresponding to 
the $\xi_i$'s. It may be necessary to shift a little the rectangular tiles so 
that the balls around the $\xi_i$'s do not meet their boundaries (this is the 
only reason of $r_j$ not being exactly dyadic).\\ 
We use \eqref{eq:concentrinmoyen} for the terms involving $\psi^{\rm in}_i$ and 
\eqref{eq:concentroutmoyen} for those with $\psi^{\rm out}_j.$ Since 
$\chi$ vanishes at $r=0$ we have $|\chi(r,\cdot)| \leq r \|\nabla 
\chi\|_\infty$ and in the dual norm we may crudely estimate
$$
\|r\tilde\Psi_\xi^*\chi^2 \psi_j^{\rm out}\|_{W^{1,\infty}} \leq \frac{C}{\ell_\eps}\leq C \logeps,\qquad 
\|r\tilde\Psi_\xi^*\chi^2 \psi_i^{\rm in}\|_{W^{1,\infty}} \leq \frac{C}{\rho_\eps} \leq C \eps^{-\frac23},
$$
so that we finally obtain
\begin{equation}\label{eq:tarik3}
T_3 \leq C \left(\eps^\frac13\logeps^Ce^{C\Sigma_a^r}+ 
\eps^\frac34\logeps^2\Ew\big( u,\H\setminus \cup_{i=1}^n 
B(\xi_i,\le/2\big)\right) \leq C\eps^\frac13 \logeps^Ce^{C\Sigma_a^r}.
\end{equation}
Combining \eqref{eq:tarik1},\eqref{eq:tarik2} and \eqref{eq:tarik3} in \eqref{eq:norway} we derive
\begin{equation}\label{eq:allezpresque}
    \Ew(u,\{r\geq \tilde r\}) \geq H_\eps(\xi)+T_4 -C\Big( 
\eps^\frac13\logeps^Ce^{C\Sigma_a^r}
 + \tilde r^\frac32 \logeps^\frac12 + \frac{\eps}{\tilde r}\logeps^2
\Big),
\end{equation}
and combining the latter with the definition of $\Sigma_\xi$
yields the upper bound
\begin{equation}\label{eq:arena}
    \Ew(u,\{r\leq \tilde r\}) +T_4 \leq \Sigma_\xi + C\Big(  
\eps^\frac13\logeps^Ce^{C\Sigma_a^r}
 + \tilde r^\frac32 \logeps^\frac12 + \frac{\eps}{\tilde r}\logeps^2
\Big).
\end{equation}
On the other hand, by definition of $r_\mathcal{S}$ we also have the lower bound
\begin{equation}\label{eq:matamor}
\Ew(u,\{r\leq r_\mathcal{S}\}) \geq   \Ew(u,\{r_\mathcal{S}/2 \leq r \leq r_\mathcal{S}\}) \geq \frac{\pi}{24}r_\mathcal{S}\logeps.
\end{equation}
The comparison of \eqref{eq:arena} specified for $\tilde r = r_\mathcal{S}$ and \eqref{eq:matamor} 
leads to the conclusion that 
\begin{equation}\label{eq:bieneu}
    r_\mathcal{S} \leq C \big( \frac{\Sigma_\xi}{\logeps} + \eps^\frac13\logeps^Ce^{C\Sigma_a^r}\big).
\end{equation}

{\noindent \bf  Step 7 : improved closeness and upper energy bounds.} 
We now choose $\tilde r =C \big( \Sigma_\xi/\logeps + \eps^\frac13\logeps^Ce^{C\Sigma_a^r}\big)$ in 
\eqref{eq:arena} to obtain
\begin{equation}\label{eq:finienfin}
 \Ew(u,\{r\leq \tilde r\})+ 
 \int_{\H_{\xi,\rho}}\!\!\!\!\!  r \left[  e_\eps(|u|) + 
\big|\frac{j(u)}{|u|}-j_*\big|^2\right] \chi^2 \leq 
 C \big( \Sigma_\xi + \eps^\frac13\logeps^Ce^{C\Sigma_a^r}\big).
\end{equation}
The same estimate with $\tilde r$ replaced by half its value, combined with the fact 
that in the integral of \eqref{eq:finienfin} the integrand is pointwise dominated by the one of $\Ew$, 
allows, in view of the first term of \eqref{eq:finienfin}, to get rid of $\chi^2$ in the 
integrand and conclude that  
\begin{equation}\label{eq:finienfinetnon}
 \Ew(u,\{r\leq \tilde r\})+ 
 \int_{\H_{\xi,\rho}}\!\!\!\!\!  r \left[  e_\eps(|u|) + 
\big|\frac{j(u)}{|u|}-j_*\big|^2\right] \leq 
 C \big( \Sigma_\xi + \eps^\frac13\logeps^Ce^{C\Sigma_a^r}\big),
\end{equation}
which yields \eqref{eq:dsd2}, for a suitable value of $C_1$, by taking $\rho=\re=\eps^\frac23.$ 
Note that combining the lower bound \eqref{eq:premestim} (with the error terms now controlled) 
with the lower bounds \eqref{eq:lowerblocwb} (used for $\rho=\re=\eps^\frac23$ and for all except one 
$i$) and Lemma \ref{lem:hors_disque0}, we also obtain, in view of the  
definition of $\Sigma_a$,  
\begin{equation}\label{eq:upperboundlocal}
\Ew(u,B(\xi_i,\re)) \leq  r(\xi_{i})\Big(\pi 
\log\frac{\re}{\eps} + \gamma\Big) + C\big(\Sigma_\xi + 
\eps^\frac13\logeps^Ce^{C\Sigma_a^r}\big),
\end{equation}
so that 
\begin{equation}\label{eq:patissier}
\E(u,B(\xi_i,\re)) \leq  \pi 
\log\frac{\re}{\eps} + \gamma  + C\big(\Sigma_\xi + 
\eps^\frac13\logeps^Ce^{C\Sigma_a^r}\big),
\end{equation}
for any $1\leq i \leq n.$  
Inequality \eqref{eq:sigmaxicontrolbad} is a direct consequence of \eqref{eq:aprochexi} and the 
explicit form of $H_\eps.$ Finally, it remains to improve the local estimate 
\eqref{eq:concentrinmoyen} to the more
global one \eqref{eq:dsd1}. For that purpose, it suffices to use a (possibly countable) 
partition of unity, exactly as we did in Step 6, and to rely either on \eqref{eq:dsd1} 
or on Theorem \ref{thm:B.1bis}. By the chain rule, the $W^{-1,1}$ norms after the test function is
multiplied by the functions of the partition are increased at most by a factor being the sup norm 
of the gradients of the partition, which in our case is bounded by $C\logeps.$ 
Estimate \eqref{eq:dsd1} then follows by summation as in \eqref{eq:tarik3}, and 
adapting $C_1$ if necessary.\hfill\qed

\medskip
\noindent{\bf Proof of Proposition \ref{prop:closecore}.} 
First notice that in view of Remark \ref{rem:coincide} and estimate \ref{eq:dsd2}, it suffices
to establish an inequality like \eqref{eq:inthecore} only on each of the balls $B(\xi_i,\eps^\frac23)$.
The proof is very reminiscent of Step 5 in the proof of Proposition \ref{prop:stronglocal}.
We decompose the energy as in \eqref{eq:funddecomp}, but with $j_*$ replaced by $j^\natural$ (here and in the
sequel for simplicity we write $j^\natural$ in place of $j^\natural(u^*_\xi)$):
\begin{equation}\label{eq:funddecomp2}
 e_\eps(u) = \frac12 |j^\natural|^2 +j^\natural\big(\frac{j(u)}{|u|}-j^\natural\big) +  
e_\eps(|u|) + \frac12 \big| \frac{j(u)}{|u|}-j^\natural\big|^2.
\end{equation}
Recall that $\re = \eps^\frac23$ and let $\chi_i$ be a cut-off function with compact support in $B(\xi_i,2\re)$ and 
such that $\chi_i\equiv 1$ on $B(\xi_i,\re)$ and $|\nabla \chi_i|\leq C/\re.$ 
On one hand, similar to \eqref{eq:upperboundlocal} we have the upper bound
\begin{equation}\label{eq:upperboundlocalbis}
    \int r e_\eps(u) \chi_i \leq  r(\xi_{i})\Big(\pi 
\log\frac{2\re}{\eps} + \gamma\Big) + C\big(\Sigma_a^r + 
\eps^\frac13\logeps^Ce^{C\Sigma_a^r}\big).
\end{equation}
On the other hand, by direct computation and the definition \eqref{eq:defreps} of $r_\xi$ we 
have the lower bound
\begin{equation}\label{eq:directcomp}
    \int r\frac{|j^\natural|^2}{2}\chi_i \geq \pi r(\xi_i)\log \frac{2\re}{r_\xi} - C \geq 
    \pi r(\xi_i)\log \frac{\re}{\eps} - C (\Sigma_a^r+ \log\logeps).
\end{equation}
To conclude, it suffices then to control the cross-term in \eqref{eq:funddecomp2}.  
We write
$$
\int r j^\natural\big(\frac{j(u)}{|u|}-j^\natural\big) \chi_i = 
\int r j^\natural\big(j(u)-j^\natural\big) \chi_i +\int r j^\natural\frac{j(u)}{|u|}(|u|-1)\chi_i
$$
and then for arbitrary $\kappa \in \R,$ 
\begin{equation*}\begin{split}
    \int r j^\natural\big(j(u)-j^\natural\big) \chi_i &= \int \nabla^\perp (r\Psi^\natural_\xi-\kappa)\big(j(u)-j^\natural\big) \chi_i\\
                                                      &= -\int (r\Psi^\natural_\xi-\kappa) {\rm curl}(j(u)-j^\natural)\chi_i -\int (r\Psi^\natural_\xi-\kappa)  (j(u)-j^\natural)\nabla^\perp \chi_i.
\end{split}\end{equation*}
Finally, we split
$$
\int (r\Psi^\natural_\xi-\kappa)  (j(u)-j^\natural)\nabla^\perp \chi_i = 
\int (r\Psi^\natural_\xi-\kappa)  (\frac{j(u)}{|u|}-j^\natural)\nabla^\perp \chi_i
+\int (r\Psi^\natural_\xi-\kappa)  \frac{j(u)}{|u|}(|u|-1)\nabla^\perp \chi_i.
$$
We choose $\kappa$ to be the mean value of $r\Psi^\natural_\xi$ over the support of $\nabla^\perp\chi_i$, and therefore
in view of the logarithmic nature of $\Psi^\natural_\xi$ we have the upper bound $|r\Psi^\natural_\xi-\kappa|\leq C$
on the support of $\nabla^\perp \chi_i.$ 
As in Proposition \ref{prop:stronglocal}, by Cauchy-Schwarz and the $L^\infty$ bound on $j^\natural$, we estimate
$$
|\int r j^\natural\frac{j(u)}{|u|}(|u|-1)\chi_i| \leq \frac{C}{r_\xi}\eps\Ew(u,B(\xi_i,2\re)) \leq C
$$
and
$$
|\int (r\Psi^\natural_\xi-\kappa)  \frac{j(u)}{|u|}(|u|-1)\nabla^\perp \chi_i| \leq C\frac{\eps}{\re}\Ew(u,B(\xi_i,2\re)) \leq C.
$$
Next, we have
\begin{equation*}\begin{split}
    |\int (r\Psi^\natural_\xi-\kappa)  
    (\frac{j(u)}{|u|}-j^\natural)\nabla^\perp \chi_i| &
    \leq C\| \frac{j(u)}{|u|}-j^\natural\|_{L^2({\rm supp}(\nabla \chi_i))} \|\nabla \chi_i\|_{L^2}\\
                                                                                         &\leq C\big(\Sigma_a^r + 
\eps^\frac13\logeps^Ce^{C\Sigma_a^r}\big)^\frac12\\
&\leq C(\Sigma_a^r+1).
\end{split}\end{equation*}
For the last term, we write
\begin{equation*}\begin{split}
|\int (r\Psi^\natural_\xi-\kappa) {\rm curl}(j(u)-j^\natural)\chi_i| 
&\leq C \|2J(u)-{\rm curl}j^\natural\|_{W^{-1,1}(B(\xi_i,2\re))} \| (r\Psi^\natural_\xi-\kappa)\chi_i\|_{W^{1,\infty}}\\ 
&\leq C r_\xi \frac{1}{r_\xi} \leq C,
\end{split}\end{equation*}
where we have used \eqref{eq:dsd1} and the fact that by construction
$$
\|{\rm curl}j^\natural-2\pi\sum_{i=1}^n \delta_{\xi_i}\|_{\dot W^{-1,1}(B(\xi_i,2\re))}\leq C r_\xi.
$$
The conclusion follows.\hfill\qed

\noindent{\bf Proof of Proposition \ref{prop:maintermdynamics}.}
    Since $j^\natural(u_\xi^*)$ is not sufficiently regular across the boundaries of the sets $\mathcal{C}_i$, 
    defined after \eqref{eq:defreps}, the computation which follows
    \eqref{eq:expandF} does not hold as is with $X$ replaced by $j^\natural(u_\xi^*)$ and we need instead to 
    divide the integration domain $\H$ into the union of the pieces $\mathcal{C}_i$ and of the complement of 
    this union. Performing the integration by parts then imply (only) some boundary terms,
    which actually end up in justifying \eqref{eq:Fjoli} provided ${\rm curl X}$ is
    understood in a weak sense according to \eqref{eq:curljnatural} and ${\rm
    div}(rX)$ according to \eqref{eq:divjnatural}, namely
    \begin{equation}\label{eq:dynamintegral}
    \mathcal{F}\big(j^\natural(u_\xi^*),\varphi\big)=-\sum_{i=1}^n
    \int_{\partial\mathcal{C}_i} |j(u_\xi^*)|j(u_\xi^*)\cdot \nabla \varphi.  
\end{equation}
For each fixed $i$, to compute the boundary term on $\partial\mathcal{C}_i$ we use a reference
polar frame $(\rho,\theta)$ centered at $\xi_i$. First by construction of
    $\mathcal{C}_i$ and \eqref{eq:dAa}-\eqref{eq:Aaz} we have
    \begin{equation}\label{eq:presquecercle}
    \rho = r_\xi +O\big(r_\xi^2\log(r_\xi)\big),
    \end{equation}
    so that $\mathcal{C}_i$ is close to being a circle, and then by \eqref{eq:dAa}, \eqref{eq:Aaz} and \eqref{eq:presquecercle},
    \begin{equation}\label{eq:jnaturalexpand}
    j^\natural(u_\xi^*) = \frac{e_\theta}{r_\xi} + \sum_j \mathbb{J}\nabla_{a_j}H_\eps(\xi_1,\cdots,\xi_n) + O(\Sigma_a^r + \log\logeps) \qquad \text{on } \partial\mathcal{C}_i,
    \end{equation}
    where the main error term, of order $\Sigma^r_a + \log\logeps$, comes from the
    difference between $\logeps$ (as appearing in the definition of $H_\eps$) and $\log r_\xi$ (from the value of 
    $\Psi^*_\xi$ on $\mathcal{C}_i$). The computation of the right-hand-side of \eqref{eq:dynamintegral} is 
    then a direct consequence of \eqref{eq:presquecercle} and \eqref{eq:jnaturalexpand}, with a cancellation at 
    main order since $\frac{e_\theta}{\rho}$ integrates to zero on a circle. The actual 
    details are left to the reader.  \hfill \qed

\noindent{\bf Proof of Proposition \ref{prop:momentumapprox}.}
Since estimate \eqref{eq:dsd1} is only valid for $r$ not too close to zero, we shall split
$\chi_\eps$ into two pieces. More precisely, we write $1 = \Psi_1 + \Psi_2$ where $\Psi_1$ is
supported in $r \leq 2\tau \equiv 2C_1(\Sigma_a^r+1)/\logeps$, $\Psi_2$ is supported in $r\geq \tau$, 
and $|\nabla \Psi_1|+|\nabla \Psi_2| \leq 10/\tau.$ Using \eqref{eq:dsd1} we immediately
obtain
$$
\left|\int Ju r^2\chi_\eps\Psi_2 - \pi\sum_{i=1}^n r(\xi_i)^2\right| \leq C_1 \eps\logeps^{C_1}e^{C_1\Sigma_a^r}\|r^2\chi_R\Psi_2\|_{W^{1,\infty}} \leq C \eps\logeps^{C_1}e^{C_1\Sigma_a^r}  R_\eps^2. 
$$
To estimate the part involving $\Psi_1$, and in particular the singularity at $r=0,$ we use
Theorem \ref{thm:B.1ter} (more precisely its higher dimensional extension - see e.g. \cite{JeSo}) 
in the 3D cylinder in cartesian coordinates  corresponding to $r\leq 2\tau$ and
$|z| \leq 2R_\eps.$ Writing back its statement in cylindrical coordinates yields
\begin{equation}\begin{split}
    \left|\int Ju r^2\chi_\eps\Psi_1 \right| \leq& C  \ \frac{\Ew(u,\{r \leq 2\tau\})}{\logeps} 
\|r\chi_\eps\Psi_1\|_\infty\\
&\ + C \eps^\frac{1}{24}(1+ \frac{\Ew(u,\{r \leq 2\tau\})}{\logeps} )(1+C\tau^2R_\eps)  
\|r\chi_\eps\Psi_1\|_{\mathcal{C}^{0,1}}\\
\leq& \ C  \frac{1+\Sigma_a^r}{\logeps}\tau \\
\leq& \ C  \frac{(1+\Sigma_a^r)^2}{\logeps^2,}
\end{split}\end{equation}
provided $\eps$ is  required to be sufficiently small. By summation we obain \eqref{eq:approxJ}.
To obtain \eqref{eq:approxddtJ}, we notice that in the expansion \eqref{eq:expandF} the terms for 
which the derivatives of $\varphi$ fall onto $r^2$ exactly cancel (that would correspond without cut-off to
the conservation of the momentum) and the remaining ones (where the derivatives fall onto $\chi_\eps$) are 
pointwise bounded by $C e_\eps(u)r|\nabla \chi_\eps|$, so that the conclusion follows by integration 
and \eqref{eq:dsd2}.\hfill\qed 

\medskip
\noindent{\bf Proof of Proposition \ref{prop:firstspeedbound}.} 
In this proof  $\|\cdot\|$
is understood to mean $\dot W^{-1,1}(\Omega_0)$ and $|\cdot|$  refers to
the Euclidean norm on $\H.$

We write 
\begin{equation}\label{eq:decompjac}\begin{split}
    \|Ju_\eps^{s_1}-Ju_\eps^{s_2}\| &\leq
    \|Ju_\eps^{s_1}-\pi\sum_{i=1}^n \delta_{\xi_i(s_1)}\|
+
\|Ju_\eps^{s_2}-\pi\sum_{i=1}^n \delta_{\xi_i(s_2)}\|
+
\|\pi\sum_{i=1}^n (\delta_{\xi_i(s_1)}-\delta_{\xi_i(s_2)})\|\\
&\leq r_\xi^{s_1}+r_\xi^{s_2} + \pi \sum_{i=1}^n
|\xi_i(s_1)-\xi_i(s_2)|.
\end{split}\end{equation}
If $C_6$ is chosen sufficiently large, it follows from the separation assumption \eqref{eq:bonneechelle}, 
the finite speed of propagation of $\lf_\eps$, and the definition of $S_{\rm stop}$, that
$$
\xi_i(s) \in B(a_{i,\eps}(s_1),\frac{\rho_{\rm min}}{4}) \qquad \forall\ s \in [s_1,s_2].
$$
Let 
\begin{equation}\label{eq:defvarphi}
\varphi(x) = \sum_{i=1}^n
\frac{(x-a_{i,\eps}(s_1))\cdot(\xi_i(s_2)-\xi_i(s_1))}{|\xi_i(s_2)-\xi_i(s_1)|} \chi\Big(
|x-a_{i,\eps}(s_1)|\Big),
\end{equation}
where $\chi\in\mathcal{C}^\infty(\R^+,[0,1])$ is such that 
$\chi\equiv 1$ on $[0,\rho_{\rm min}/4],$ 
$\chi\equiv 0$ on  $[\frac{\rho_{\rm min}}{2},+\infty).$ 
By construction and the definition of $\rho_{\rm min}$, we have $\varphi \in
\mathcal{D}(\Omega_0)$ and it follows that 
\begin{equation}\label{eq:bornedistpoints}
\begin{split}
\pi \sum_{i=1}^n |\xi_i(s_1)-\xi_i(s_2)| 
&= \langle \pi\sum_{i=1}^n
(\delta_{\xi_i(s_2)}-\delta_{\xi_i(s_1)}), \varphi \rangle\\
&\leq (r_\xi^{s_1} + r_\xi^{s_2})\|\varphi\|_{W^{1,\infty}} + 
\langle Ju_\eps^{s_2} - J u_\eps^{s_1} ,\varphi\rangle.
\end{split}
\end{equation}
Combining this with \eqref{eq:decompjac}, we conclude that 
\begin{equation}\label{eq:forlan}
    \| Ju_\eps^{s_2} - Ju_\eps^{s_1}\| \leq C( r^{s_1}_\xi + r^{s_2}_\xi)  +\langle
    Ju_\eps^{s_2} - J u_\eps^{s_1} ,\varphi\rangle. 
\end{equation}
By \eqref{eq:fondam1},
\begin{equation}\label{eq:norway0}
    \left|\langle Ju_\eps^{s_2} - J u_\eps^{s_1} ,\varphi\rangle\right| \leq
    \frac{1}{\logeps}|\int_{s_1}^{s_2} \mathcal{F}(\nabla u_\eps^s,\varphi)\, ds|.
\end{equation}
Recall that
\begin{equation}\label{eq:defFanouveau}
\mathcal{F}(\nabla u_\eps^s, \varphi)
:= -\int_\H \eps_{ij} \frac{\partial_k r}{r} (\partial_j u_\eps^s  ,\partial_k u_\eps^s)\partial_i\varphi 
+\int_\H \eps_{ij}(\partial_ju_\eps^s,\partial_k u_\eps^s)\partial_{ik}\varphi,
\end{equation}
and that by \eqref{eq:defvarphi} we have
$$
\partial_{ik}\varphi \equiv 0 \qquad \text{on } \cup_i B(a_{i,\eps}(s_1),\frac{\rho_{\rm min}}{4}).
$$
Since $|\nabla \varphi|\leq C$ and $|D^2\varphi| \leq C/\rho_{\rm min}$, we have
$$
\left|\int_\H \eps_{ij} \frac{\partial_k r}{r} (\partial_j u_\eps^s  ,\partial_k u_\eps^s)\partial_i\varphi \right|
\leq C\Ew(u_\eps^s) \leq C\logeps
$$
and by \eqref{eq:dsd2}, \eqref{eq:sigmaxicontrol} and \eqref{eq:stop}
$$
\left| \int_\H \eps_{ij}(\partial_ju_\eps^s,\partial_k u_\eps^s)\partial_{ik}\varphi\right| \leq 
\frac{C}{\rho_{\rm min}}\left( \Sigma^0+r_a^s\logeps+\log\logeps\right) \leq C \logeps.
$$
Going back to \eqref{eq:norway0} we thus obtain
\begin{equation}\label{eq:norway1}
 \left|\langle Ju_\eps^{s_2} - J u_\eps^{s_1} ,\varphi\rangle\right| \leq
    C|s_1-s_2|
\end{equation}
and therefore
\begin{equation}\label{eq:norway2}
    \|Ju_\eps^{s_1}-Ju_\eps^{s_2}\| \leq C(r^{s_1}_\xi + r^{s_2}_\xi + |s_1-s_2|). 
\end{equation}
It remains to estimate $r_\xi^{s_2}.$
For that purpose, we write 
\begin{equation}\label{eq:rxis}\begin{split}
    r_a^{s_2} &= \| J u_\eps^{s_2} - \pi \sum_{i=1}^n \delta_{a_{i,\eps}(s_2)}\| \\ 
              &\leq \|Ju_\eps^{s_2} - J u_\eps^{s_1}\| + \|Ju_\eps^{s_1} -\pi\sum_{i=1}^n
    \delta_{a_{i,\eps}(s_1)}\| + \|\pi\sum_{i=1}^n
    (\delta_{a_{i,\eps}(s_1)}-\delta_{a_{i,\eps}(s_2)})\|\\
    &\leq r_a^{s_1} + C\bigl( |s_1-s_2| + r_\xi^{s_1} + r_\xi^{s_2} \bigr),
\end{split}\end{equation}
where we have used \eqref{eq:norway2} and \eqref{eq:bornedistpoints}. 
By the definition \eqref{eq:defrxis} of $r_\xi^{s_2}$ and \eqref{eq:rxis} we obtain
\begin{equation}\label{eq:rxis2}
    \begin{split}
        r_\xi^{s_2} &\leq C_1\eps\logeps^{C_1}e^{C_1(\Sigma^0+r_a^{s_2}\logeps)}\\
                    &\leq r_\xi^{s_1}e^{C(|s_2-s_1|+2\eps^\frac56)\logeps}\\
                    &\leq r_\xi^{s_1}(1+C\left(|s_2-s_1|+\eps^\frac56\right)\logeps), 
\end{split}
\end{equation}
where we have used the fact that $|s_2-s_1| \leq \logeps^{-1}$ by assumption. It remains to prove
the last assertion of the statement, namely that if $r_a^{s_1} < \rho_{\rm min}/16$ then $S_{\rm stop}
\geq s_1+(C_6\sqrt{\logeps})^{-1}.$ By definition of $S_{\rm stop},$ the latter follows easily from \eqref{eq:rxis}
and \eqref{eq:rxis2}, increasing the value of $C_6$ if necessary. 
\qed

\medskip
\noindent{\bf Proof of Proposition \ref{prop:controldiscrep}.}
The proof follows very closely the strategy  used in \cite{JeSm3} Proposition 7.1.
By \eqref{eq:rxiequiv} and the definition of $S$ we first remark that 
\begin{equation}\label{eq:rxidouble}
    r_\xi^\tau \leq 2 r_\xi^s \qquad \forall \tau \in [s,S].
\end{equation}
Next, note that 
\begin{equation}\begin{split}
r^S_a - r^s_a &= 
\| Ju_\eps^S - \pi \sum_{i=1}^n  \delta_{a_{i,\eps}(S)}\|
-
\| Ju_\eps^s - \pi \sum_{i=1}^n  \delta_{a_{i,\eps}(s)}\|\\
&\leq 
\pi \sum_{i=1}^n \left(|\xi_i(S)-a_{i,\eps}(S)| - |\xi_i(s) - a_{i,\eps}(s)| \right)\ + \ r^S_\xi + r^s_\xi
\\
&\leq 
\pi \sum_{i=1}^n  \nu_i \cdot \bigl(\xi_i(S)- \xi_i(s)  + a_{i,\eps}(s)- a_{i,\eps}(S) \bigr)\ + \ 
r^S_\xi + r^s_\xi
\label{eq:finland0}
\end{split}
\end{equation}
for $\nu_i = \frac {\xi_i(S)-a_{i,\eps}(S)}{|\xi_i(S)-a_{i,\eps}(S)|}$ (unless $\xi_i(S)- a_{i,\eps}(S)=0$, in which
case $\nu_i$ can be any unit vector).
We let 
\[
    \varphi(x) = \sum_i  \nu_i \cdot(x - a_{i,\eps}(s)) \chi(|x - a_{i,\eps}(s)|)
\]
for $\chi\in C^\infty(\R^+, [0,1])$ such that $\chi \equiv 1$ on $[0,\frac 14 \rho_{\rm min}]$ and $\chi\equiv 0$ 
on $(\frac12 \rho_{\rm min},\infty)$. 
It follows from \eqref{eq:dansboule} that 
\[
    \pi \sum_{i=1}^n  \nu_i \cdot \bigl(\xi_i(S)- \xi_i(s)  + a_{i,\eps}(s)- a_{i,\eps}(S) \bigr)\
=
\pi \sum_{i=1}^n  \Bigl[ \varphi(\xi_i(S))- \varphi(\xi_i(s))  - \varphi(a_{i,\eps}(S)) + \varphi(a_{i,\eps}(s)) \Bigr],
\]
so that \eqref{eq:finland0} and the definition of $r^S_\xi$ imply that 
\begin{equation}\label{eq:finland1}
r^S_a-r^s_a \leq \langle  \varphi , Ju_\eps^S - Ju_\eps^s \rangle  \ -
\ \pi\sum_{i=1}^n \Bigl[\varphi(a_{i,\eps}(S)) - \varphi(a_{i,\eps}(s))\Bigr]
+\ C(r^S_\xi + r^s_\xi).
\end{equation}
Our main task in the sequel is therefore to provide an estimate for the quantity 
$ \langle  \varphi , Ju_\eps^S - Ju_\eps^s \rangle $.
By \eqref{eq:fondam1} (and taking into account the $\logeps$ change of scale in time) we have
\begin{equation}\label{eq:finland2}
\langle \varphi , Ju_\eps^S - Ju_\eps^s \rangle 
= \int_s^S  \frac{1}{\logeps} \mathcal{F}(\nabla u_\eps^\tau ,\varphi)\, d\tau.
\end{equation}
In the sequel for the ease of notation we write $u$ in place of $u_\eps^\tau$, for $\tau \in [s,S].$ 
Similar to what we did in \eqref{eq:funddecomp}, and in view of the definition \eqref{eq:defF} of $\mathcal{F},$
we decompose here
\begin{equation}\begin{split}
    \left( \partial_j u,\partial_k u \right) &= \partial_j |u|\partial_k |u| + \frac{j(u)_j}{|u|}\frac{j(u)_k}{|u|}\\
                                             &= \partial_j |u|\partial_k |u| + \left(\frac{j(u)}{|u|}-j^\natural\right)_j\left(\frac{j(u)}{|u|}-j^\natural\right)_k\\
                                             &\quad +  \left(\frac{j(u)}{|u|}-j^\natural\right)_j\left(j^\natural\right)_k +\left(\frac{j(u)}{|u|}-j^\natural\right)_k\left(j^\natural\right)_j\\
                                             &\quad + \left(j^\natural\right)_j\left(j^\natural\right)_k,
\end{split}\end{equation}
where
$$
j^\natural \equiv j^\natural(u_{\xi(\tau)}^*).
$$
Hence,
\begin{equation}\label{eq:splitmotion}
    \mathcal{F}(\nabla u,\varphi) = \mathcal{F}(j^\natural,\varphi)+ \sum_{p=1}^4 T_p
\end{equation}
where
$$
T_1 := -\int_\H \eps_{ij} \frac{\partial_k r}{r} \left[
  \partial_j |u|\partial_k |u| + \left(\frac{j(u)}{|u|}-j^\natural\right)_j\left(\frac{j(u)}{|u|}- 
  j^\natural\right)_k\right]\partial_i\varphi, 
$$
$$
T_2 := \int_\H \eps_{ij}\left[\partial_j |u|\partial_k |u| +
\left(\frac{j(u)}{|u|}-j^\natural\right)_j\left(\frac{j(u)}{|u|}-j^\natural\right)_k\right]\partial_{ik}\varphi,
$$
$$
T_3 := -\int_\H \eps_{ij} \frac{\partial_k r}{r}  \left[
    \left(\frac{j(u)}{|u|}-j^\natural\right)_j\left(j^\natural\right)_k +
\left(\frac{j(u)}{|u|}-j^\natural\right)_k\left(j^\natural\right)_j \right] 
\partial_i\varphi, 
$$
and
$$
T_4 := \int_\H \eps_{ij} \left[
    \left(\frac{j(u)}{|u|}-j^\natural\right)_j\left(j^\natural\right)_k +
\left(\frac{j(u)}{|u|}-j^\natural\right)_k\left(j^\natural\right)_j \right] 
\partial_{ik}\varphi.
$$
By Proposition \ref{prop:maintermdynamics} we already know that
$$
\left| \mathcal{F}(j^\natural,\varphi) - \sum_{i=1}^n 
\mathbb{J}\nabla_{a_i}H_\eps(\xi_1(\tau),\cdots,\xi_n(\tau))\cdot
\nabla\varphi(\xi_i(\tau)) \right| \leq C_3\left( 
    \Sigma^0+r_a^\tau\logeps + \log\logeps\right),
$$
moreover since $S \leq S_{\rm stop}$ we have for any $i=1,\cdots,n,$
$$
\left|  
\nabla_{a_i}H_\eps(\xi_1(\tau),\cdots,\xi_n(\tau))- 
\nabla_{a_i}H_\eps(a_{1,\eps}(\tau),\cdots,a_{n,\eps}(\tau))
\right| \leq C \logeps r_a^\tau
$$
and since $\varphi$ is affine there, 
$$
\nabla\varphi(\xi_i(\tau)) = \nabla\varphi(a_{i,\eps}(\tau)), 
$$
so that after integration and using the fact that the points $a_{i,\eps}$ evolve
according to the ODE $\lf_\eps$ we obtain  
$$
\left|\int_s^S \frac{1}{\logeps}\mathcal{F}(j^\natural,\varphi)\, d\tau - 
\pi\sum_{i=1}^n \Bigl[\varphi(a_{i,\eps}(S)) - \varphi(a_{i,\eps}(s))\Bigr]\right|
\leq C(S-s)\left(r_a^s + \frac{\Sigma^0}{\logeps} +
\frac{\log\logeps}{\logeps}\right). 
$$
Now that we have accounted for the main order, we need to control all the terms
$T_p$ (at least integrated in time between $s$ and $S$).
We begin with the terms $T_1$ and $T_2$ for which we already have good estimates 
(pointwise in time) thanks to Proposition \ref{prop:stronglocal} and Proposition
\ref{prop:closecore}. Indeed, by Proposition \ref{prop:closecore} and since $|\nabla
\varphi| \leq C$, we have
\begin{equation}\label{eq:boundT1}
    \frac{|T_1|}{\logeps} \leq C\left( r_a^s + \frac{\Sigma^0}{\logeps} +
    \frac{\log\logeps}{\logeps}\right). 
\end{equation}
By Proposition \ref{prop:stronglocal} and \eqref{eq:sigmaxicontrol}, and since
$|D^2\varphi| \leq C\sqrt{\logeps}$, we have
\begin{equation}\label{eq:boundT2}
    \frac{|T_2|}{\logeps} \leq C\left( r_a^s + \frac{\Sigma^0}{\sqrt{\logeps}}
        +\logeps^{-\frac32} + \sqrt{\logeps}
    \Big|P_\eps(u)-P\big(a_{1,\eps}(\tau)\cdots,a_{n,\eps}(\tau)\big)\Big| \right).
\end{equation}
In order to deal with the last term involving $P$ and $P_\eps$, recall first that $P$ 
is preserved by the flow $\lf_\eps$, so that
\begin{equation}\label{eq:sweden0}
P\big(a_{1,\eps}(\tau),\cdots,a_{n,\eps}(\tau)\big) =
P\big(a_{1,\eps}(0),\cdots,a_{n,\eps}(0)\big),
\end{equation}
and that
$P_\eps$ is almost preserved by $\gp$, as expressed by \eqref{eq:approxddtJ}, so that
\begin{equation}\label{eq:sweden1}
    \left| P_\eps(u)- P_\eps(u_\eps^0)\right| \leq C\frac{S_{\rm stop}}{\logeps^3}\left(
    \Sigma^0+ \sqrt{\logeps}\right) \leq  C S_{\rm stop} \logeps^{-\frac52}
\end{equation}
where we have used the rough bound \eqref{eq:sigmaxicontrolbad} for $\Sigma_\xi$, the
rough estimate $r_a^\tau \leq C/\sqrt{\logeps}$ which follows from the definition of
$S_{\rm stop}$, and where we have taken into account the factor $\logeps^{-1}$ which 
arises from the change of time scale which we have here with respect to the one of 
\eqref{eq:fondam1}.
On the other hand, at the initial time by \eqref{eq:approxJ} and the bound
\eqref{eq:boundinit} on $r_a^0$ and $\Sigma^0$ we have
\begin{equation}\label{eq:sweden2}
    \left|P_\eps(u_\eps^0)-  P\big(\{a_{i,\eps}(0)\}\big)\right| \leq 
    \left|P_\eps(u_\eps^0)-  P\big(\{\xi_{i}(0)\}\big)\right| + Cr_a^0 \leq 
    C\left(r_a^0+\frac{1}{\logeps^2}\right). 
\end{equation}
In total, similar to \eqref{eq:boundT1} we obtain
\begin{equation}\label{eq:boundT2bis}
    \frac{|T_2|}{\logeps} \leq C\left( r_a^s+
        \frac{\Sigma^0+r_a^0\logeps}{\sqrt{\logeps}} +
    \frac{\log\logeps}{\logeps}\right),
\end{equation}
where we have absorbed some of the above error terms by the term $\log\logeps/\logeps.$ 
We decompose
$$
T_3 = T_{3,1}+T_{3,2}+T_{3,3}
$$
where
$$
T_{3,1}:= -\int_\H \eps_{ij} \frac{\partial_k r}{r}  \left[
    \left(\frac{j(u)}{|u|}-j^\natural\right)_j\left(j^\natural-j^\natural(u_{\xi(s)}^*)\right)_k +
\left(\frac{j(u)}{|u|}-j^\natural\right)_k\left(j^\natural-j^\natural(u_{\xi(s)}^*)\right)_j \right] 
\partial_i\varphi, 
$$
$$
T_{3,2}:= -\int_\H \eps_{ij} \frac{\partial_k r}{r}  \left[
    \left(j(u)-j^\natural\right)_j\left(j^\natural(u_{\xi(s)}^*)\right)_k +
\left(j(u)-j^\natural\right)_k\left(j^\natural(u_{\xi(s)}^*)\right)_j \right] 
\partial_i\varphi, 
$$
and
$$
T_{3,3}:= -\int_\H \eps_{ij} \frac{\partial_k r}{r}  \left[
    \left(\frac{j(u)}{|u|}\right)_j\left(j^\natural(u_{\xi(s)}^*)\right)_k +
\left(\frac{j(u)}{|u|}\right)_k\left(j^\natural(u_{\xi(s)}^*)\right)_j \right](1-|u|) 
\partial_i\varphi, 
$$
and accordingly we decompose $T_4 = T_{4,1}+T_{4,2}+T_{4,3}.$   
We first deal with $T_{3,3}$ and $T_{4,3}$, where invoking the inequality
\begin{equation}\label{eq:degiorgi}
\frac{|j(u)|}{|u|}(1-|u|) \leq \frac{\eps}{2} \left(\frac{|j(u)|^2}{|u|^2} + \frac{(1-|u|)^2}{\eps^2}\right) \leq C\eps e_\eps(u)
\end{equation}
combined with the  global $\logeps$ bound on the energy and the $L^\infty$ bound
$|j^\natural| \leq Cr_\xi^{-1} \leq C\eps^{-1}/\logeps^{C_1}$ we directly infer (increasing $C_1$ if necessary) that
\begin{equation}\label{eq:T4T6}
    \frac{T_{3,3}+T_{4,3}}{\logeps} \leq \frac{C}{\logeps}.
\end{equation}
We next turn to the terms $T_{3,1}$ and $T_{4,1}$, for which we rely on Proposition \ref{prop:firstspeedbound}  
and the definition of $S$ to get the upper bound
\begin{equation}\label{eq:traveldist}
    \sum_{i=1}^n |\xi_i(s)-\xi_i(\tau)| \leq C\left( r_\xi^s + \frac{(r_\xi^s)^2}{\eps}\right) 
    \leq C \frac{(r_\xi^s)^2}{\eps}. 
\end{equation}
Using the almost explicit form of $j^\natural$ (more precisely \eqref{eq:Aa}, \eqref{eq:dAa} and the definition
of the cut-off at the scale $r_\xi^s$), we compute that
\begin{equation}\label{eq:diffcore}
    \int_{{\rm supp}(\varphi)} \left| j^\natural-j^\natural(u_{\xi(s)}^*)\right|^2 \leq C
    (1+\log\left(\frac{\frac{(r_\xi^s)^2}{\eps}}{r_\xi^s}\right)) \leq C\left( \Sigma^0+r_a^s\logeps + \log\logeps\right). 
\end{equation}
The previous inequality combined with Proposition \ref{prop:closecore} and the Cauchy-Schwarz inequality 
then yields   
$$
\frac{T_{3,1}}{\logeps} \leq C\left( r_a^s + \frac{\Sigma^0}{\logeps} + \frac{\log\logeps}{\logeps}\right).
$$
For $T_{4,1}$, since the integration domain does no longer contain the cores we obtain the stronger estimate
\begin{equation}\label{eq:diffcorebis}
    \int_{{\rm supp}(D^2\varphi)} \left| j^\natural-j^\natural(u_{\xi(s)}^*)\right|^2 \leq C \frac{(\sum_{i=1}^n |\xi_i(s)-\xi_i(\tau)|)^2}{\rho_{\rm min}^2} \leq
    C\eps^\frac23\logeps,  
\end{equation}
where we have used \eqref{eq:traveldist} and \eqref{eq:securitybounds} for the last inequality. Using once more the Cauchy-Schwarz inequality, combined here with 
Proposition \ref{prop:stronglocal} (or even simply the crude $\logeps$ global energy bound) and the $L^\infty$ bound
$|D^2\varphi| \leq C\sqrt{\logeps}$ we obtain
$$
\frac{T_{4,1}}{\logeps} \leq C\eps^\frac13 \sqrt{\logeps}.
$$
At this stage we are left to estimate $T_{3,2}$ and $T_{4,2}$, which we will only be able to do after integration
in time. To underline better the time dependence, it is convenient here to write $j^\natural_\tau$ in place of
$j^\natural$ and $j^\natural_s$ in place of $j^\natural(u_\xi^s).$ The main ingredient in the argument is then 
to perform a Helmholtz type decomposition of $j(u)-j^\natural_\tau.$ More precisely, we first fix a cut-off
function $\chi$ with compact smooth support $B$ in $\{r \geq \frac{r_0}{8}\}$,  which is identically equal to $1$ on the 
support of $\varphi$ and which satisfies $|\nabla \chi| \leq C$ (its only aim is to get rid of boundary terms,
of spatial infinity, and of the singularity at $r=0$). For every $\tau \in [s,S],$ we then set
\begin{equation}\label{eq:helmholtz}
    \chi(j(u)-j^\natural_\tau) = \nabla f^\tau + \frac{1}{r}\nabla^\perp g^\tau\qquad\text{in } B,
\end{equation}
where
$f^\tau$ and $g^\tau$ are the unique solutions of the Neumann 
\begin{equation}\label{eq:ellipf}
    \begin{array}{ll}
        {\rm div}\left( r \nabla f^\tau\right) = {\rm div}\left( \chi r(j(u)-j^\natural_\tau)\right) & \text{ in } B,\\
        \partial_n f^\tau = 0 & \text{ on } \partial B,\\
        \int_B f^\tau = 0,
    \end{array}
\end{equation}
and Dirichlet
\begin{equation}\label{eq:ellipg}
    \begin{array}{ll}
        -{\rm div}\left( \frac{1}{r} \nabla g^\tau\right) = {\rm curl}\left( \chi (j(u)-j^\natural_\tau)\right) & \text{ in } B,\\
        g^\tau = 0 & \text{ on } \partial B
    \end{array}
\end{equation}
boundary value problems. By construction,
\begin{equation}\label{eq:reste1}
\int_s^S \frac{T_{3,2}}{\logeps}\,d\tau = -\frac{1}{\logeps}\int_B \eps_{ij} \frac{\partial_k r}{r}
\left[(\nabla F + \frac{\nabla^\perp G}{r})_j(j^\natural_s)_k 
+(\nabla F + \frac{\nabla^\perp G}{r})_k(j^\natural_s)_j \right] \partial_i\varphi 
\end{equation}
and
\begin{equation}\label{eq:reste2}
\int_s^S \frac{T_{4,2}}{\logeps}\,d\tau = -\frac{1}{\logeps}\int_B \eps_{ij} 
\left[(\nabla F + \frac{\nabla^\perp G}{r})_j(j^\natural_s)_k 
+(\nabla F + \frac{\nabla^\perp G}{r})_k(j^\natural_s)_j \right] \partial_{ik}\varphi 
\end{equation}
where $F=\int_s^S f^\tau\,d\tau$ and $G=\int_s^S g^\tau\,d\tau.$ \\
Integrating \eqref{eq:ellipf}, 
we split $F=F_1+F_2+F_3$ where   
$$
 {\rm div}\left( r \nabla F_p\right) = L_p \text{ in } B,\qquad  \partial_n F_p = 0  \text{ on } \partial B, \qquad\int_B F_p = 0,
$$
for $p=1,2,3,$ and where
$$
L_1:= \chi\int_s^S {\rm div}(rj(u)),\qquad L_2:=r\nabla \chi \cdot \int_s^S \frac{j(u)}{|u|}(|u|-1),\qquad
L_3:= r\nabla\chi\cdot\int_s^S (\frac{j(u)}{|u|}-j^\natural_\tau).
$$
Similarly, integrating \eqref{eq:ellipg} we split $G=G_1+G_2+G_3$ where  
$$
{\rm div}\left( \frac{1}{r} \nabla G_p\right) = M_p \text{ in } B,\qquad  M_p = 0  \text{ on } \partial B,
$$
for $p=1,2,3,$ and where
$$
M_1:= \chi\int_s^S {\rm curl}(j(u)-j^\natural_\tau),\qquad M_2:=\nabla^\perp \chi \cdot \int_s^S \frac{j(u)}{|u|}(|u|-1),\qquad
M_3:= \nabla^\perp\chi\cdot\int_s^S (\frac{j(u)}{|u|}-j^\natural_\tau).
$$
Before we state precise bounds for each of them, we note that it should be clear at this stage that all the terms $L_p$ and $M_p$ 
are small in some sense, except perhaps for the term $L_1$ which requires some more explanation.
For that last term, we rely on the continuity equation (and this is the main reason for the integration in time)
\begin{equation}\label{eq:continuity}
\partial_t |u|^2 = \frac{2}{r}{\rm div}(rj(u))
\end{equation}
which is a consequence of $\gp$, and from which we infer that
\begin{equation}\label{eq:L1new}
    L_1 = \eps \frac{\chi}{\logeps}\left[ \frac{(|u|^2-1)}{\eps} \right]^S_s,
\end{equation}
so that
$$
\|L_1\|_{L^2} \leq   C\frac{\eps}{\sqrt{\logeps}}\leq C\frac{(S-s)}{\sqrt{\logeps}}\frac{\eps^2}{(r_\xi^s)^2}
\leq C(S-s) \logeps^{-2C_1},
$$
where we have used the definitions of $S$ and $r_\xi^s.$
Regarding $L_2$ and $M_2$, using \eqref{eq:degiorgi} we easily obtain 
$$
\|L_2\|_{L^1}+\|M_2\|_{L^1} \leq   C\eps(S-s)\logeps.
$$
For $L_3$ and $M_3$, we use the fact that $\nabla \chi$ lives away from the cores
so that  Proposition  \ref{prop:stronglocal} and Proposition 
\ref{cor:apprixgood} yield (we bound the
terms involving $P$ in \eqref{eq:sigmaxicontrol} exactly as we did to simplify \eqref{eq:boundT2} into \eqref{eq:boundT2bis})
$$
\|L_3\|_{L^2} +\|M_3\|_{L^2}\leq C(S-s) \left( r_a^s\sqrt{\logeps}+
\Sigma^0+r_a^0\logeps + \logeps^{-1}\right)^{\frac12} \leq C(S-s).
$$
Finally, regarding $M_1$, we have on one side using \eqref{eq:dsd1} and \eqref{eq:rxidouble} 
$$
\|M_1\|_{W^{-1,1}} \leq C(S-s)r_\xi^s,
$$
and on the other side using the pointwise inequality $Ju \leq Ce_\eps(u)$ for
an arbitrary function $u$ and the global energy bound 
$$
\|M_1\|_{L^1} \leq C(S-s)\logeps.
$$
By interpolation, it follows that for any $1<p<2$
$$
\|M_1\|_{W^{-1,p}} \leq C_p(S-s)(r_\xi^s)^\theta\logeps^{1-\theta},
$$
where $\frac{1}{p} = \theta + \frac{1-\theta}{2}.$
These bounds on $L_i$ and $M_i$ turn into bounds on $F_i$ and $G_i$, since by standard elliptic 
estimates we have
\begin{equation}\label{eq:ellipticestim}
    \begin{array}{ll}
        \|\nabla F_1\|_{L^p} \leq C_p \|L_1\|_{L^2} & \text{ for all } 1 \leq p<+\infty,\\ 
        \|\nabla F_2\|_{L^p} + \|\nabla G_2\|_{L^p} \leq C_p (\|L_2\|_{L^1}+\|M_2\|_{L^1}) &
    \text{ for all } 1\leq p<2,\\
        \|\nabla F_3\|_{L^p} + \|\nabla G_3\|_{L^p} \leq C_p (\|L_3\|_{L^2}+\|M_3\|_{L^2}) &
    \text{ for all } 1\leq p<+\infty,\\
    \|\nabla G_1\|_{L^p} \leq C_p \|M_1\|_{W^{-1,p}} & \text{ for all } 1<p<2.\\ 
    \end{array}
\end{equation}
    
To estimate \eqref{eq:reste1}, we then simply input \eqref{eq:ellipticestim} into \eqref{eq:reste1}  
where we use the H\"older inequality with $j^\natural_s$ estimated in $L^{p'}$ (and all the other 
weights other than $F_i$ or $G_i$ in $L^\infty$). The largest contribution arises from $G_1$ since
$\|j^\natural_s\|_{L^{p'}} \simeq (r^s_\xi)^{-\theta}$ when $p<2$ and the final $\logeps^{-1^-}$ 
bound is obtained by choosing $p$ arbitrarilly close to $1$ (and hence $\theta$ arbitrarilly close to $1$). 
For the terms involving $p>2$ we use the straightforward bound $\|j^\natural_s\|_{L^{p'}} \leq C.$
    
The estimate of \eqref{eq:reste2} is at first sight slightly more difficult since $\partial_{ik}\varphi$ 
is diverging like $\sqrt{\logeps}$ whereas in \eqref{eq:reste1} $\partial_{i}\varphi$ was bounded
in absolute value. On the other hand, the integrand only lives on the support of $D^2\varphi$, which
is both away from the cores and of Lebesgues measure of order $\logeps^{-1}.$ More precisely, for
$F_1$ and $G_1$ we rely exactly on the same estimate as in \eqref{eq:ellipticestim}, whereas, since 
the forcing terms $L_2, L_3, M_2$ and $M_3$ have a support disjoint from that of $D^2\varphi$, it follows
by elliptic regularity that
\begin{equation}\label{eq:ellipcoeur}
\|\nabla F_i\|_{L^\infty({\rm supp}(D^2\varphi))} \leq  C\|\nabla F_i\|_{L^1}, \quad  
\|\nabla G_i\|_{L^\infty({\rm supp}(D^2\varphi))}  \leq  C\|\nabla G_i\|_{L^1}, \quad i=2,3.
\end{equation}
We then combine \eqref{eq:ellipcoeur} with our previous estimate \eqref{eq:ellipticestim}, and therefore
in the H\"older estimate of \eqref{eq:reste2} for these four terms we can take $p=\infty$ and hence
$p'=1.$  Finally, regarding $j^\natural_s$ we have for $p=\infty$ 
$$
\|j^\natural_s\|_{L^{1}({\rm supp}(D^2\varphi))}\leq \|j^\natural_s\|_{L^{\infty}({\rm supp}(D^2\varphi))} \mathcal{L}^2( {\rm supp}(D^2\varphi))) \leq C\rho_{\rm min},
$$ 
and for $p<2$ 
$$
\|j^\natural_s\|_{L^{p'}({\rm supp}(D^2\varphi))}
\leq C (\rho_{\rm min})^{-\theta}
$$ 
where $\frac{1}{p} = \theta + \frac{1-\theta}{2}.$ The conclusion then follows by summation.\hfill\qed

\medskip
\noindent{\bf Proofs of Theorem \ref{thm:main} and Theorem \ref{thm:main_asympt}}
Theorem \ref{thm:main} follows very directly from Proposition \ref{prop:controldiscrep}. Indeed, the iterative
use of Proposition \ref{prop:controldiscrep} leads to a discrete Gronwall inequality which is a forward Euler scheme 
for the corresponding classical (continuous) Gronwall inequality, and the latter has convex solutions which are therefore
greater than their discrete equivalent.  The actual details can be taken almost word for word from the ones used 
in \cite{JeSm3} Proof of Theorem 1.3, and are therefore not repeated here. 

Finally, Theorem \ref{thm:main_asympt} is also easily deduced from Theorem \ref{thm:main}. The only point which 
deserves additional explanation is the fact that in the assumptions of Theorem \ref{thm:main_asympt} only local
norms $\|\cdot\|_{\dot W^{-1,1}(\Omega)}$ with $\Omega$ being of compact closure in the interior of $\H$ are used
whereas the definition of $r_a^0$ for Theorem \ref{thm:main} involves the unbounded set $\Omega_0.$ As the proof
of Proposition \ref{prop:stronglocal} shows (more precisely its Step 6), the closeness estimates expressed in \eqref{eq:dsd1}
(which hold in expanding domains whose union ends up covering the whole of $\H$ as $\eps$ tends to zero) only require 
a first localisation estimate in a neigborhood of size $1/\sqrt{\logeps}$ of the points $a_{i,\eps}$, which is of course
implied by the assumptions of Theorem \ref{thm:main_asympt}. 
\hfill \qed

\numberwithin{thm}{section}
\numberwithin{prop}{section}
\numberwithin{lem}{section}
\numberwithin{rem}{section}
\appendix

\section{Vector potential of loop currents}\label{subsect:canon}

In the introduction we have considered the inhomogeneous Poisson equation
$$
\left\{
\begin{array}{ll}
\displaystyle -{\rm div} \left(\frac{1}{r} \nabla \left( r A_a\right)
\right) =  2\pi\delta_{a} &\qquad \text{in } \H,\\
A_a  = 0 & \qquad \text{on } \H.
\end{array}
\right.
$$  
Its integration is classical (see e.g. \cite{Jackson}) and yields 
$$
A_a(r,z) = \frac{r(a)}{2} \int_0^{2\pi} \frac{\cos(t)}{\sqrt{r(a)^2+r^2 +
(z-z(a))^2-2r(a)r\cos(t)}}\, dt,
$$
which in turn simplifies to
$$
A_a(r,z) = \sqrt{\frac{r(a)}{r}} \frac{1}{k} \left[
(2-k^2)K(k^2)-2E(k^2)\right]
$$
where 
$$
k^2 = \frac{4r(a) r}{ r(a)^2 + r^2 + (z-z(a))^2 + 2 r(a) r}
$$
and where $E$ and $K$ denote the complete elliptic integrals of first and second
kind respectively (see e.g. \cite{Abra}). Note that $A_{\lambda a}(\lambda r,
\lambda z)= A_a(r,z)$ for any $\lambda >0$ and that we have 
the asymptotic expansions \cite{Abra} of the complete elliptic integrals as 
$s\to 1$
:
$$
K(s) = -\frac{1}{2}\log(1-s)\left(1 + \frac{1-s}{4}\right) + \log(4) + O(1-s),
$$
$$
E(s ) = 1 - \log(1-s)\frac{1-s}{4} + O(1-s),
$$
and similarly for their derivatives.  For $(r,z) \in \H \setminus
\{a\},$ direct computations therefore yield  
\begin{equation}\label{eq:Aa}
A_a(r,z) =  \left( \log (\tfrac{r(a)}{\rho}) + 
3 \log(2)
-2\right) +  O\left(\tfrac{\rho}{r(a)}
|\log (\tfrac{\rho}{r(a)})|\right)\qquad \text{as } \tfrac{\rho}{r(a)} \to
0,
 \end{equation}
and
\begin{equation}\label{eq:dAa}
\partial_{\rho} A_a = -\frac{1}{\rho} +  O\left(\tfrac{1}{r(a)}
\right)\qquad \text{as } \tfrac{\rho}{r(a)} \to
0,
\end{equation}
where $\rho := |a-(r,z)|.$

\noindent
Concerning the asymptotic close to $r=0,$ we have
\begin{equation}\label{eq:Aaz}
A_a(r,z) \simeq \frac{rr(a)^2}{r(a)^3+|z|^3} \qquad \text{as } \frac{r}{r(a)} 
\to 0. 
\end{equation}

\subsection{Singular unimodular maps}\label{appendice:canon}
When $a = \{a_1,\cdots, a_n\}$ is a family of $n$ distinct 
points 
in
$\H$, we define the function $\Psi^*_a$ on
$\H_{a}:=\H\setminus a$ by
$$
\Psi^*_a = \sum_{i=1}^n A_{a_i},
$$
so that
$$
\left\{
\begin{array}{ll}
\displaystyle -{\rm div} \left( \frac{1}{r} \nabla (r \Psi^*_a)\right) = 2\pi
\sum_{i=1}^n
\delta_{a_i} &\qquad \text{on } \H,\\
\displaystyle \Psi^*_a = 0 & \qquad \text{on } \partial\H.
\end{array}
\right.
$$  
Up to a constant phase shift, there exists a unique unimodular map $u^*_a 
\in \mathcal{C}^\infty(\H_a,S^1)\cap W^{1,1}_{\rm
loc}(\H,S^1)$ such that
$$
r (iu^*_a,\nabla u^*_a) = rj(u^*_a) = -\nabla^\perp (r\Psi^*_a).
$$   
In the sense of distributions in $\H$, we have
$$
\left\{
\begin{array}{ll}   
\displaystyle {\rm div}(rj(u^*_a)) & = 0\\
\displaystyle {\rm curl}(j(u^*_a)) & = 2\pi \sum_{i=1}^n \delta_{a_i}.
\end{array}
\right.
$$
Let 
$$
 \rho_a :=\frac14\min\Big( \min_{i\neq j}|a_i-a_j|, \min_i 
r(a_i)\Big),
$$ 
for $\rho\leq \rho_a$ we set 
$$\H_{a,\rho} :=
\H \setminus \cup_{i=1}^n B(a_i,\rho).
$$
\begin{lem}\label{lem:hors_disque0}
Under the above assumptions we have
$$
\int_{\H_{a,\rho}} \frac{|j(u^*_a)|^2}{2} r\,drdz = \pi \sum_{i=1}^n r(a_i) 
\Big[\log\big(\tfrac{r(a_i)}{\rho}\big)+
\sum_{j\neq i}A_{a_j}(a_i) +
\big(3\log(2)-2\big) + 
O\big(\tfrac{\rho}{\rho_a}\log^2(\tfrac{\rho}{\rho_a}) \big)\Big]. 
$$
\end{lem}
\begin{proof}
 We have the pointwise equality
 $$
 |j(u^*_a)|^2 r = \frac{1}{r}|\nabla^\perp (r\Psi^*_a)|^2 = 
\frac{1}{r}|\nabla (r\Psi^*_a)|^2,
 $$
 so that after integration by parts
 $$
 \int_{\H_{a,\rho}} \frac{|j(u^*_a)|^2}{2} r\,drdz = - 
\frac12 \int_{\H_{a,\rho}} {\rm 
div}\left(\frac{1}{r}\nabla(r\Psi^*_a)\right)r\Psi^*_a \,drdz + 
\frac12 \int_{\partial 
\H_{a,\rho}}  \Psi^*_a \nabla(r\Psi^*_a)\cdot \vec n, 
 $$
 and the first integral of the right-hand side in the previous identity 
vanishes by definition of $\Psi^*_a$ and $\H_{a,\rho}.$ We next decompose 
the
boundary integral as
$$
\frac12 \int_{\partial 
\H_{a,\rho}}  \Psi^*_a \nabla(r\Psi^*_a)\cdot \vec n = 
\frac12 \sum_{i,j,k=1}^n \int_{\partial B(a_i,\rho)} A_{a_j} 
\nabla(rA_{a_k})\cdot \vec 
n,
$$
and for fixed $i,j,k$ we write
$$
A_{a_j} \nabla(rA_{a_k})\cdot \vec 
n =  \left(-A_{a_j} 
\partial_\rho A_{a_k} r  + A_{a_j}A_{a_k}n_r\right).  
$$
Using \eqref{eq:Aa}, we have 
$$
|\int_{\partial B(a_i,\rho)} A_{a_j}A_{a_k}n_r | \leq 
r(a_i)O\Big(\tfrac{\rho}{\rho_a}\log^2(\tfrac{\rho}{\rho_a})\Big).
$$
When $i=j=k,$ we have by \eqref{eq:Aa} and \eqref{eq:dAa}
$$
-\frac12 \int_{\partial B(a_i,\rho)} A_{a_j} 
\partial_\rho A_{a_k} r = \pi r(a_i) \left( \log(\frac{r(a_i)}{\rho}) + 
3\log(2) -2 + 
O\Big(\tfrac{\rho}{\rho_a}\log(\tfrac{\rho}{\rho_a})\Big)\right) 
$$
while when $i=k\neq j$ we have
$$
-\frac12 \int_{\partial B(a_i,\rho)} A_{a_j} 
\partial_\rho A_{a_k} r = \pi r(a_i) \left( A_{a_j}(a_i)  + 
O\Big(\tfrac{\rho}{\rho_a}\Big)\right).
$$
Finally, when $i\neq k$ we have
$$
\big|\frac12 \int_{\partial B(a_i,\rho)} A_{a_j} 
\partial_\rho A_{a_k} r \big| \leq r(a_i) 
O\Big((\tfrac{\rho}{\rho_a})^2 
\log(\tfrac{\rho}{\rho_a})\Big).
$$
The conclusion follows by summation.
\end{proof}

If we next fix some constant $K_0>0$ and we assume that 
the points $a_i$ are
of the form  
$$
a_{i} := \Big(r_0+ \frac{r(b_i)}{\sqrt{\logeps}} , 
z_0 + \frac{z(b_i)}{\sqrt{\logeps}} \Big),\qquad i=1,\cdots,n,
$$
for some $r_0>0$, $z_0\in \R$ and $n$ points
$\{b_1,\cdots,b_n\} \in \R^2$ which satisfy 
$$
\max_i |b_i| \leq K_0,\qquad
\text{and}\qquad \min_{i\neq j} {\rm dist}(b_i,b_j)\geq \frac{1}{K_0},
$$
we directly deduce from Lemma \ref{lem:hors_disque0}, \eqref{eq:Aa} and 
\eqref{eq:dAa} :
\begin{lem}\label{lem:sanscoeur}
Under the above assumptions we have
\begin{equation*}
\begin{split}
\int_{\H_{a,\rho}} \frac{|j(u^*_a )|^2}{2} r\,drdz =&\ \pi n r_0 \Big(
|\log\rho| + n\log r_0 + n\big(3\log(2)-2\big)
+ \tfrac{n-1}{2}\log|\log\eps|\Big) \\
&+  \pi r_0 \Big( \sum_{i} \frac{r(b_i)}{r_0}\frac{|\log 
\rho|}{\sqrt{|\log\eps|}} -  \sum_{i\neq j} 
\log |b_i-b_j|\Big)\\
&+ O_{K_0,r_0} \Big( \frac{1}{\sqrt{|\log\eps|}}\Big).
\end{split}
\end{equation*}
\end{lem}

\section{Jacobian and Excess for 2D Ginzburg-Landau functional}
\label{sect:B}

For the ease of reading, we recall in this appendix a few results from 
\cite{JeSo}, \cite{JeSp} and \cite{JeSp2} which we use in our work.

\begin{thm}[Thm 1.3 in \cite{JeSp} - {\bf Lower energy 
bound}]\label{thm:B.1}
There exists an absolute constant $C>0$ such that for any $u \in 
H^1(B_r,\C)$ satisfying $\|Ju-\pi\delta_0\|_{\dot W^{-1,1}(B_r)} < r/4$ we 
have
$$
\E(u,B_r) \geq \pi \log\frac{r}{\eps} + \gamma -\frac{C}{r} \left( 
\eps\sqrt{\log \tfrac{r}{\eps}} + \|Ju-\pi\delta_0\|_{\dot 
W^{-1,1}(B_r)}\right).
$$
\end{thm}

\begin{thm}[from Thm 1.1 in \cite{JeSp} - {\bf Jacobian estimate without 
vortices}]\label{thm:B.1bis}
There exists an absolute constant $C>0$ with the following property. If  
$\Omega$ is a bounded domain, $u \in H^1(\Omega,\C)$,  
$\eps \in (0,1]$ and $\E(u,\Omega)< \pi \logeps$, then 
$$
\Big\| Ju \Big\|_{\dot W^{-1,1}(\Omega)} \leq  \eps C \E(u,\Omega)\exp\Big( 
\frac{1}{\pi}\E(u,\Omega)\Big). 
$$
\end{thm}

\begin{thm}[Thm 2.1 in \cite{JeSo} - {\bf Jacobian estimate 
with vortices}]\label{thm:B.1ter} There exists an absolute constant $C>0$ with 
the following property. If  
$\Omega$ is a bounded domain, $u \in H^1(\Omega,
\C)$, and $\varphi \in 
\mathcal{C}^{0,1}_c(\Omega)$, then for any $\lambda \in (1,2]$ and any $\eps 
\in (0,1)$,
$$
\Big| \int_\Omega \varphi Ju \, dx \Big| \leq \pi d_\lambda \|\varphi\|_\infty 
+ \|\varphi\|_{\mathcal{C}^{0,1}}h^\eps(\varphi,u,\lambda)
$$
where
$$
d_\lambda = \Big\lfloor \frac{\lambda}{\pi} \frac{\E(u,{\rm spt}(\varphi))}{\logeps} 
\Big\rfloor,
$$
$\lfloor x \rfloor$ denotes the greatest integer less than or equal to $x$, and
$$
h^\eps(\varphi,u,\lambda) \leq C 
\eps^\frac{\lambda-1}{12\lambda}\big(1+\frac{\E(u,{\rm spt}(\varphi))}{\logeps}\big) 
\big(1+\mathcal{L}^2({\rm spt}(\varphi)\big).
$$
\end{thm}

\begin{thm}[Thm 1.2' in \cite{JeSp2} - {\bf Jacobian localization 
for a vortex in a ball}]\label{thm:B.2}
There exists an absolute constant $C>0$, such that for any $u \in 
H^1(B_r,\C)$ satisfying 
$$\|Ju-\pi\delta_0\|_{\dot W^{-1,1}(B_r)} < r/4,$$
if we write
$$
\Xi = \E(u,B_r) - \pi \log \frac{r}{\eps}
$$
then there exists a point $\xi \in B_{r/2}$ such that
$$
\|Ju - \pi \delta_\xi \|_{\dot W^{-1,1}(B_r)} \leq 
 \eps C(C+\Xi)\left[(C+\Xi)e^{\Xi/\pi} + \sqrt{\log \frac{r}{\eps}}\right].
$$
\end{thm}

\begin{thm}[Thm 3 in \cite{JeSp2} - {\bf Jacobian localization for many 
vortices}]\label{thm:B.3}
Let $\Omega$ be a bounded, open, simply connected subset of $\R^2$ with 
$\mathcal{C}^1$ boundary. There exists constants $C$ and $K$, depending on 
${\rm diam}(\Omega)$, with the following property: For any $u \in 
H^1(\Omega,\C)$, if there exists $n\geq0$ distinct points $a_1,\cdots,a_n$ in 
$\Omega$ and $d \in \{\pm 1\}^{n}$ such that
$$
\|Ju-\pi\sum_{i=1}^n d_i \delta_{a_i}\|_{\dot W^{-1,1}(\Omega)} \leq 
\frac{\rho_a}{Kn^5},
$$
where 
$$
\rho_a := \frac14{\rm min}_i\left\{ {\rm min}_{j\neq i}|a_i-a_j|, {\rm 
dist}(a_i,\partial\Omega)\right\},
$$
and if in addition $\E(u,\Omega) \geq 1$ and
$$
\frac{n^5}{\rho_a}\E(u,\Omega) + \frac{n^{10}}{\rho_a^2}\sqrt{\E(u,\Omega)} 
\leq \frac{1}{\eps},
$$
then there exist $\xi_1,\cdots,\xi_d$ in $\Omega$ such that
$$
\|Ju - \pi \sum_{i=1}^n d_i \delta_{\xi_i} \|_{\dot W^{-1,1}(\Omega)} \leq 
 C\eps \left[n(C+\Xi_\Omega^\eps)^2e^{\Xi_\Omega^\eps/\pi} + 
(C+\Xi_\Omega^\eps)\frac{n^5}{\rho_a} + \E(u,\Omega)\right],
$$
where
$$
\Xi_\Omega^\eps := \E(u,\Omega)-n(\pi\log \frac{1}{\eps} + \gamma) +- \pi 
\Big( \sum_{i\neq j} d_id_j \log|a_i-a_j| + \sum_{i,j}
d_id_j H_\Omega(a_i,a_j)\Big)$$
and $H_\Omega$ is the Robin function of $\Omega.$
\end{thm}


\begin{thebibliography}{xxx}
 
\bibitem{Abra} M. Abramowitz and I. Stegun, 
{\it Handbook of mathematical functions}, U.S. Government Printing Office, 
Washington D.C., 1964. 

\bibitem{BBH} F. Bethuel, H. Brezis and F. H\'elein, {\it
Ginzburg-Landau vortices}, Birkh\"auser, Boston, 1994.

\bibitem{BeGrSm} F. Bethuel, P. Gravejat and D. Smets, {\it Stability in the 
energy space for chains of solitons of the one-dimensional Gross-Pitaevskii 
equation}, Ann. Inst. Fourier, in press, 2013.
 
\bibitem{BOS-VR}  F. Bethuel, G. Orlandi and D. Smets, {\it Vortex rings for 
the Gross-Pitaevskii equation}, J. Eur. Math. Soc. (JEMS) {\bf 6} (2004), 
17--94.  
 
\bibitem{CaMa}  D. Benedetto, E. Caglioti and C. Marchioro, {\it On the motion 
of a vortex ring with a sharply concentrated vorticity}, Math. Methods Appl. 
Sci. {\bf 23} (2000),+ 147--168. 
 
\bibitem{Dyson} F. W. Dyson, {\it The potential of an anchor ring}, 
Phil. Trans. R. Soc. Lond. A {\bf 184} (1893), 43--95. 
 
\bibitem{Hel1} H. Helmholtz, {\it Über Integrale der hydrodynamischen 
Gleichungen welche den Wirbelbewegungen entsprechen}, J. Reine Angew. Math. {\bf 
55} (1858), 25--55.

\bibitem{Hel2} H. Helmholtz (translated by P. G. Tait), {\it On the integrals of the
hydrodynamical equations which express vortex-motion}, Phil. Mag. {\bf 
33}
(1867), 485–512.

\bibitem{Hicks} W. M. Hicks, {\it On the mutual threading of vortex rings}, 
Proc. Roy. Soc. London A {\bf 102} (1922), 111-131.

\bibitem{Jackson} J. D. Jackson, Classical Electrodynamics, John Wiley \& Sons 
Ltd, New York, 1962. 

\bibitem{JeSm3} R. L. Jerrard and D. Smets, {\it Vortex dynamics for the two 
dimensional non homogeneous Gross-Pitaevskii equation}, Annali Scuola Normale 
Sup. Pisa Cl. Sci. {\bf 14} (2015), 729-766.

\bibitem{JeSo}  R. L. Jerrard and H. M. Soner, {\it The 
Jacobian and the Ginzburg-Landau energy}, Calc. Var. Partial Differential 
Equations {\bf 14} (2002), 151--191.

\bibitem{JeSp} R. L. Jerrard and D. Spirn, {\it Refined Jacobian estimates for 
Ginzburg-Landau functionals}, Indiana Univ. Math. J. {\bf 56} (2007), 135--186.

\bibitem{JeSp2} R. L. Jerrard and D. Spirn, {\it Refined Jacobian estimates and 
Gross-Pitaevsky vortex dynamics}, Arch. Ration. Mech. Anal. {\bf 190} (2008), 
 425--475. 

\bibitem{KuMaSp} M. Kurzke, J.L. Marzuola and D. Sprin, {\it Gross-Pitaevskii vortex
motion with critically-scaled inhomogeneities}, preprint, arXiv:1510.08093v2.
 
 
\bibitem{Love} A. E. H. Love, {\it On the motion of paired vortices with a 
common axis}, Proc. London Math. Soc. {\bf 25} (1893), 185--194.

\bibitem{MaNe} C. Marchioro and P. Negrini, {\it On a dynamical system related 
to fluid mechanics}, NoDEA Nonlinear Diff. Eq. Appl. {\bf 6} (1999), 473--499. 

\bibitem{MaMeTs} Y. Martel, F. Merle and T.-P. Tsai, {\it Stability and 
asymptotic stability in the energy space of the sum of N solitons for 
subcritical gKdV equations},  Comm. Math. Phys. {\bf 231} (2002), 347--373. 

\end{thebibliography}
\end{document}